\documentclass[11pt]{article}

\usepackage{mathpazo}
\usepackage{amsfonts}
\usepackage{amsmath}
\usepackage{graphicx}
\usepackage{latexsym}
\usepackage{mathabx}
\usepackage{bm,booktabs,tabu,tabularx,tensor}

\usepackage[margin=1.05in]{geometry}
  
\usepackage[T1]{fontenc}
\usepackage{times}
\usepackage{color,graphicx}
\usepackage{array}
\usepackage{enumerate}
\usepackage{amsmath}
\usepackage{amssymb}
\usepackage{amsthm}
\usepackage{pgfplots}
\usepackage{pgf}
 
\usepackage{amsthm}

\usepackage{xcolor}
\usepackage{makeidx}
\usepackage{caption}
\usepackage{subcaption}
\usetikzlibrary{shapes,snakes}

\usepackage[pdftex,linktocpage=true,colorlinks,citecolor=blue,linkcolor=blue,pagebackref]{hyperref}
\usepackage{hyperref}
\usepackage{cleveref}

\numberwithin{equation}{section}

\newtheorem{conjecture}{Conjecture}

\newtheorem{lemma}{Lemma}
\newtheorem{theorem}{Theorem}
\newtheorem{corollary}{Corollary}
\newtheorem{remark}{Remark}
\newtheorem{proposition}{Proposition}

\newtheorem{example}{Example}

\begin{document}
\title{Weakly Hadamard diagonalizable graphs and Quantum State Transfer}

\author{Darian McLaren\textsuperscript{1}, Hermie Monterde\textsuperscript{2}, and Sarah Plosker\textsuperscript{3}}
\maketitle

\addtocounter{footnote}{1}
\footnotetext{Department of Mathematics and Computer Science, Brandon University,
Brandon, MB R7A 6A9, Canada; da2mclaren@uwaterloo.ca}\addtocounter{footnote}{1}
\footnotetext{Department of Mathematics, University of Manitoba, Winnipeg, MB, Canada R3T 2N2; monterdh@myumanitoba.ca}\addtocounter{footnote}{1}
\footnotetext{Department of Mathematics and Computer Science, Brandon University,
Brandon, MB R7A 6A9, Canada; Department of Mathematics, University of Manitoba, Winnipeg, MB, Canada R3T 2N2;  ploskers@brandonu.ca}

\begin{abstract} Hadamard diagonalizable graphs are undirected graphs for which the corresponding Laplacian is diagonalizable by a Hadamard matrix. Such graphs have been studied in the context of quantum state transfer. Recently, the concept of a  weak Hadamard matrix was introduced:  a $\{-1,0, 1\}$-matrix $P$ such that $PP^T$ is tridiagonal, as well as the concept of weakly Hadamard diagonalizable graphs.  We therefore naturally explore quantum state transfer in these generalized Hadamards. Given the infancy of the topic, we provide numerous properties and constructions of weak Hamard matrices and weakly Hadamard diagonlizable graphs in order to better understand them. 
\medskip

	
\end{abstract}

\section{Introduction}

Let $\mathcal{M}_n$ be the space of real $n\times n$ matrices. 
A Hadamard matrix $H\in\mathcal{M}_n$ is a matrix whose entries are either 1 or -1 and satisfies
$H^T H=n I_n$,
where $I_n$ is the $n\times n$ identity matrix. The columns of a Hadamard matrix are therefore mutually orthogonal.

The study of Hadamard diagonalizable graphs---namely, undirected graphs $X$ for which the associated Laplacian matrix  can be diagonalized by
some Hadamard matrix, was first initiated in \cite{barik2011}. More recently, such graphs have been found to be useful in the study of quantum state transfer, where one considers the time evolutionary operator $U(t)=e^{(itL(X))}$ for the graph $X$. In this line of inquiry, one is interested in finding graphs that allow for desirable types of quantum state transfer, such as perfect state transfer, pretty good state transfer, and fractional revival.   A  wide variety of new graphs that exhibit perfect state transfer, for both weighted and unweighted graphs, were found in \cite{johnston2017}, and \cite{chan2020complex} generalizes the concept of Hadamard diagonalizable graphs from real to complex Hadamard matrices, giving necessary and sufficient conditions, and examples, of $(\alpha, \beta)$-fractional revival and perfect state transfer in this setting.  

Recently, a generalization of Hadamard matrices was proposed in \cite{adm2021}, where a matrix $P\in \mathcal M_n$ is called a \emph{weak Hadamard matrix} if all entries of $P$ are from $\{-1, 0, 1\}$ and $P^TP$ is a tridiagonal matrix.
The generalization of weak Hadamard relaxes the orthogonality  condition of a Hadamard to allow any two consecutive columns to not necessarily be orthogonal. The concept was introduced to study weakly Hadamard diagonalizable graphs. Many properties of weak Hadamard matrices have yet to be explored.

Our aim herein is to explore the natural next question of quantum state transfer in weakly Hadamard diagonalizable graphs. Before doing so, it is necessary to take a step back to uncover properties of weak Hadamard matrices and weakly Hadamard diagonalizable graphs, to build up constructions and examples of such matrices and graphs, to motivate the utility of considering these graphs in the context of quantum state transfer. Indeed, if there were few instances of  weak Hadamard matrices and weakly Hadamard diagonalizable graphs, then it would be of little value to physicists to know that perfect state transfer occurs in only a handful of graphs. That being said, we frequently assume some additional structure of the columns of the weak Hadamard matrices beyond quasi-orthogonality in our study of weakly Hadamard diagonalizable graphs and quantum state transfer in such graphs (typically either assuming all vectors not equal to the all-ones vector are orthogonal to it, or assuming orthogonality between all vectors). This is due to the fact that quasi-orthogonality is not easy to exploit; in any case, over 44\% of the graphs on at most nine vertices that are weakly Hadamard diagonalizable are diagonalized by a weak Hadamard with pairwise orthogonal columns (in particular, 100\% for $n=4$ and over 88\% for $n=8$) \cite{johnston2023laplacian}.

In Section \ref{sec:preliminaries} we consider some mathematical preliminaries to our study, in particular, we consider algebraic and combinatorial properties of weak Hadamard matrices, and describe a number of methods for constructing such matrices.  In Section \ref{sec:whd} we consider spectral properties of weakly Hadamard diagonalizable  graphs. We also construct examples of weakly Hadamard diagonalizable graphs based on known ones. Finally, in Section \ref{sec:qst}, we look at quantum state transfer in weakly Hadamard diagonalizable graphs, first focusing on strong cospectrality, which is a necessary condition for many types of quantum state transfer; we then consider perfect state transfer and graph operations preserving perfect state transfer in weakly Hadamard diagonalizable graphs.

\section{Preliminaries}\label{sec:preliminaries}
In \cite{adm2021}, weak Hadamard matrices were introduced to study graphs whose Laplacian matrix is diagonalized by a weak Hadamard matrix. Theirs was the first study of this type of matrix. Therefore, many properties and constructions of weak Hadamard matrices have not yet been developed. 

\subsection{Properties of Weak Hadamard Matrices}
Here, we ask what   are the combinatorial and spectral properties of weak Hadamard matrices, and  which matrix operations preserve the property of being a weak Hadamard matrix.

The following are some easily-provable properties of weak Hadamard matrices. Since the topic of weak Hadamard matrices is in its infancy, these results have not been observed elsewhere. 
It is well-known that Hadamard matrices of order $n>2$ satisfy  $n\equiv 0$ (mod 4). In the following theorem,  we present sufficient conditions for when the same holds for weak Hadamard matrices. We will use the notation $\mathbf{1}$ to denote the all-ones vector of appropriate size. 

\begin{theorem}
\label{WHDcombi}
Let $P\in \mathcal{M}_n$ be a weak Hadamard matrix with two columns $\mathbf{x}$ and $\mathbf{z}$ that are orthogonal to $\mathbf{1}$. Then the following statements hold:
\begin{enumerate}
\item The vector $\mathbf{x}$ has an even number of non-zero entries, exactly half of which are $1$'s. 
\item If $\mathbf{x}$ has $k$ number of $1$'s and $r$ number of $0$'s, then $r=n-2k$. If $k$ is even, then $n\equiv r$ (mod 4).
\item If $\mathbf{x}$ has all entries non-zero, then $n=2k$. Moreover, $k$ is even if and only if $n\equiv 0$ (mod 4).
\item If (i) $k$ is even and $r\equiv 0$ (mod 4) or (ii) all entries of both $\mathbf{x}$ and $\mathbf{z}$ are non-zero, then $n\equiv 0$ (mod 4).
\end{enumerate}
\end{theorem}
\begin{proof}
Suppose $\mathbf{x}$ has $k$ number of $1$'s and $m$ number of $-1$'s. Since $\mathbf{1}^T\mathbf{x}=0$, we get $k=m$, and so $\mathbf{x}$ has $r=n-2k$ zero entries. In particular, if $\mathbf{x}$ has all entries non-zero, then $n=2k$. This proves (1)-(3).

Let us now prove (4). Condition (i) yields the desired conclusion by (2). Now, suppose condition (ii) holds. Then (3) implies that $n$ is even and by a suitable relabelling of the vertices of $X$, we may assume that $z=[\mathbf{1},-\mathbf{1}]^T$. As $\mathbf{z}^T\mathbf{x}=0$, we get
\begin{equation*}
\sum_{j=1}^{\frac{n}{2}}\mathbf{x}_j\ =\sum_{j=\frac{n}{2}+1}^n\mathbf{x}_j.
\end{equation*}
Assume that for the first $\frac{n}{2}$ entries of $\mathbf{x}$, $k_1$ of them are equal to 1 and $k_2$ of them are equal to $-1$. Then $\sum_{j=1}^{\frac{n}{2}}\mathbf{x}_j=k_1-k_2$. Moreover, $k-k_1$ of the latter $\frac{n}{2}$ entries of $\mathbf{x}$ are equal to 1, while $k-k_2$ of them are equal to $-1$, which gives us $\sum_{j=\frac{n}{2}+1}^{n}\mathbf{x}_j=(k-k_1)-(k-k_2)=k_2-k_1$. By the above equation, $k_1=k_2$. Since $\mathbf{x}$ has all entries non-zero, $k_1+k_2=\frac{n}{2}$, and so $n=2(k_1+k_2)=4k_1$. Thus, $n\equiv 0$ (mod 4).\end{proof}

Theorem \ref{WHDcombi}(1-3) require only that a single column $\mathbf{x}$ of $P$ has entries in $\{-1,0,1\}$. Thus, the theorem holds for a wider range of matrices: matrices formed by the Laplacian eigenvectors of trivalent graphs \cite{Caputo}. 

The rows and columns of a Hadamard matrix can be permuted, and any row or column can be multiplied by -1, and the resulting matrix is still a Hadamard matrix. Thus, it is always possible to arrange to have the first row and the first column of a Hadamard matrix to be all 1's; such a Hadamard matrix is said to be \emph{normalized}. Although the rows of a weak Hadamard matrix can be permuted, and any row or column can be multiplied by -1, column permutations in general may affect the quasi-orthogonality of the columns. One can clearly see that the identity matrix cannot be normalized, and thus not all weak Hadamard matrices can be normalized. 

The following results provide some structure of the columns of a weak Hadamard matrix for odd dimensions.

\begin{corollary}
\label{lem:OddZero}
    Let $P\in \mathcal{M}_n$ be a normalized weak Hadamard matrix of dimension $n=2k+1$ with pairwise orthogonal columns. Then each column of $P$ (other than $\mathbf{1}$) has three or more zero entries.
\end{corollary}

\begin{corollary}\label{cor:5}    
   There is no normalized weak Hadamard matrix of dimension 5 that has pairwise orthogonal columns.
\end{corollary}

For a weak Hadamard $P$, let $a_P$ be the number of blocks of any size in the block diagonal form of $P^T P$, $b_P$ be the number of blocks of size greater than or equal to 2, and $c_P$ be the number of pairs of identical columns in $P$. With these definitions we now have the following theorem which counts the number of weak Hadamard matrices equivalent under column permutations. 
 
\begin{theorem}\label{thm:EquivalentWeak}
    Let $P$ be a weak Hadamard matrix. Then the number of equivalent weak Hadamard matrices attained by permuting of the columns of $P$, which we denote by $d(P)$, is given by
    \begin{align}
        d(P) = 2^{b_P/c_P} a_P!
    \end{align}
\end{theorem}

\begin{proof}
  
 A permutation on the columns of a weak Hadamard matrix $P$ is equivalent to right multiplication on $P$ by a permutation matrix $Q$. It follows then that $Q$ is valid permutation on $P$ if only if $Q^T P^T P Q$ is tridiagonal. Therefore, instead of considering permutations on the columns of $P$ we may instead consider permutations on the rows/columns of $P^T P$. 

 We set
\begin{align}
    P^T P = \text{diag}(Y_1, Y_2, \dots, Y_{a_P} ),
\end{align}
where the dimension of each block $Y_i$ may vary from 1 to $n$. We note that there are only two types of valid permutations: (i) permuting entire blocks, or (ii) mirroring a block across its antidiagonal.

A simple counting argument verifies the formula for $d(P)$.

\end{proof}


\subsection{Construction of Weak Hadamard Matrices}\label{sec:constr}

In this subsection, we present a number of different methods of constructing weak Hadamard matrices and weak Hadamard matrices with pairwise orthogonal columns. These may prove useful for future study of this newly defined concept. 

\begin{enumerate}
    \item 
 The standard Hadamard matrices may be defined recursively through Sylvester's construction by first setting  $H_1=[1]$, $
H_2=\begin{bmatrix}
1 & 1\\
1 & -1\\
\end{bmatrix}$,
and constructing larger Hadamard matrices through the recursive equation $H_{2^k}=H_2\otimes H_{2^{k -1}}$, where $k\geq 2$ is any positive integer. Here we explore the idea of an analogue of Sylvester's construction for the case of normalized weak Hadamard matrices to construct `standard (normalized) weak Hadamard matrices'.

First note that, for any (normalized) weak Hadamard matrix $A$, the matrix $H_2\otimes A$ is also a  (normalized) weak Hadamard matrix.
    
    In subsequent sections, we will require our weak Hadamard matrices to have pairwise orthogonal columns. We therefore  seek a matrix $P$ such that $H_2\otimes P$ is a normalized weak Hadamard matrix having pairwise orthogonal columns (but is not a Hadamard matrix). 
The weak Hadamard matrix that diagonalizes $K_4\backslash e$  (the complete graph minus an edge) is 
\begin{equation}
\label{Pone}
P_1= \begin{bmatrix}
        1 & 1& 1& 0\\
        1& -1& 1& 0\\
        1 & 0 & -1 & 1\\
        1 & 0 & -1 & -1
    \end{bmatrix}.
\end{equation}
For $\ell\geq 2$, define the matrix
\begin{equation}
\label{Pell}
P_{\ell}=\left[ \begin{array}{ccccc} P_{\ell-1}&P_{\ell-1} \\ P_{\ell-1}&-P_{\ell-1} \end{array}\right].
\end{equation}
For each $\ell\geq 1$, $P_\ell$ is a normalized weak Hadamard matrix of order $2^{\ell+1}$ with pairwise orthogonal columns that is not a Hadamard matrix.

  More generally, one can ask for which matrices $A$ is it the case that $A\otimes B$ 
  is a weak Hadamard matrix whenever $B$ is a weak Hadamard matrix. The following proposition provides the answer. The result follows by considering the entries of $(A^T \otimes B^T)(A\otimes B)=A^T A \otimes B^T B$.
  
  \begin{proposition}\label{prop:tpAB}
      Let $A$ and $B$ be weak Hadamard matrices with $B$ not equal to the all-zeros matrix. Then $A\otimes B$ is a weak Hadamard matrix if and only if  the columns of $A$ are all pairwise orthogonal. 
  \end{proposition}

A corollary of the above proposition is that a Hadamard matrix tensored with a weak Hadamard matrix is   a weak Hadamard matrix; this result is implicit in \cite[Proposition 3.3]{adm2021}.
  
\item  Take any orthogonal vectors $\mathbf{a}, \mathbf{b}, \mathbf{c}, \dots$ with components in $\{-1,0,1\}$. Construct the matrix having two copies of each vector side by side as columns: $P=\begin{bmatrix} \mathbf{a},\mathbf{a},\mathbf{b},\mathbf{b},\mathbf{c},\mathbf{c},\dots\end{bmatrix}$. Such a matrix is a weak Hadamard matrix with $PP^T$ having non-zero entries only on the diagonal and super-diagonal. 

More generally, for any $n>1$, partition $\mathbb{R}^n$ into $k$ orthogonal subspaces $S_1, S_2, \dots, S_k$ where $\sum_{i=1}^k\text{dim}(S_i)=n$. From each partition select any two independent vectors $\mathbf{a}_1, \mathbf{a}_2 \in S_1$,\,  $\mathbf{b}_1, \mathbf{b}_2 \in S_2$ etc., while ensuring the vectors have components in $\{-1,0,1\}$. Then the matrix $\begin{bmatrix} \mathbf{a}_1, \mathbf{a}_2, \mathbf{b}_1,\mathbf{b}_2, \dots\end{bmatrix}$ is a weak Hadamard matrix (and invertible).

    \item Another method for constructing weak Hadamard matrices is a block design using the Williamson construction: A matrix 
            \begin{align}
                P = \begin{bmatrix}
                A&B&C&D\\
                -B&A&-D&C\\
                -C&D&A&-B\\
                -D&-C&B&A\\
                \end{bmatrix}
            \end{align}
            is a weak Hadamard matrix provided  $X^T Y =Y^T X$ for $X,Y\in\{A,B,C,D\}$ and $A^TA + B^TB + C^TC+D^TD$ is tridiagonal, with $A,B,C,D$ having entries in $\{-1,0,1\}$.

\item Similar to the Williamson construction, let $G$ and $H$ be matrices with entries in $\{-1, 0,1\}$. If $H^T H + G^T G$ is tridiagonal and $H^T G - G^T H = 0$, then $\begin{bmatrix}
        H& G\\
        G & -H
    \end{bmatrix}$
    is a weak Hadamard matrix. That being said, finding suitable matrices that satisfy the commutation relation may be troublesome. A similar construction that has easier to satisfy criteria is as follows: 
    Let $G$ and $H$ be weak Hadamard matrices of order $n$ such that $\mathbf{x}$ is the first column of both matrices, and, moreover, that $\mathbf{x}$ is orthogonal to every other column in $G$ and $H$. Then the block matrix
   $P=  
    \begin{bmatrix}
        H& X\\
        X&-G
   \end{bmatrix}$,
   with $X$ the matrix with first column $\mathbf{x}$ and all others $\mathbf{0}$, is a weak Hadamard matrix. Furthermore, if $G$ and $H$ have pairwise orthogonal columns, then so too does $P$. 

It is immediately apparent that by setting $\mathbf{x}=\mathbf{1}$ we may construct normalized weak Hadamard matrices.

\item Paley's method for constructing Hadamard matrices uses finite fields of order $q$, where $q$ is a power of an odd prime.  Let $\mathbb{Z}_q$ be the integers mod $q$.  Define
    \begin{align}
        \chi(a) = \begin{cases}
            0 & \text{if}\; a \equiv 0\\
            1 & \text{if}\; a \equiv b^2\; \text{for some non-zero} \;b \in\mathbb{Z}_q\\
            -1 & \text{otherwise}
        \end{cases}
    \end{align}
   Construct the circulant matrix $C$ by setting the $(i,j)$-th entry to be $\chi(i-j)$.
      Then the matrix $H= \begin{bmatrix}
            0&\mathbf{1}^T\\
            \mathbf{1}& C^T\\
        \end{bmatrix}$ 
    is a weak Hadamard matrix with pairwise orthogonal columns. 

\item Another construction for weak Hadamard matrices is as follows: For $n, k\in \mathbb{N}$ with $k\leq n$, define the sets $X_{2^k}$ of $2^n$ dimensional vectors whereby the set $X_{2^k}$ consists of $2^{n-k}$ vectors $\{\mathbf{x}_j\}_{j\in \mathbb{Z}_{2^{n-k}}}$, where the $i$-th component of the vector $\mathbf{x}_j$ is given by:
\begin{align}
    \mathbf{x}_j(i)=\begin{cases}
    1 & i\in [2^{k} j + 1,\dots,2^{k} j + 2^{k-1}]\\
    -1 & i\in [2^{k} j + 2^{k-1} + 1,\dots,2^{k}j + 2^{k}]\\
    0 & \text{otherwise}
    \end{cases}
\end{align}
Note that each vector $\mathbf{x}_j$ will have $2^k$ non-zero consequent entries, with an equal number of $1$'s and $-1$'s. It can then be trivially seen that each set $X_{2^k}$ consists of $2^{n-k}$ mutually orthogonal vectors, all of which are orthogonal to $\mathbf{1}$. 

It is not difficult to see that the  set $X=\bigcup_{1\leq k \leq n}X_{2^{k}}$ combined with $\mathbf{1}$ gives a collection of $2^n$ mutually orthogonal vectors. 
This collection of vectors yields the matrix of eigenvectors of $K_{2^n}\backslash e$. It follows that there exists a weak Hadamard matrix $H$ of order $2^n$ (for any $n\in \mathbb N$), such that $H$ has pairwise orthogonal columns and contains $\mathbf{1}$ as a column, but where $H$ is not a Hadamard matrix.
\end{enumerate}

\section{Weakly Hadamard diagonalizable graphs}\label{sec:whd}

We first recall some definitions from graph theory. For a simple weighted graph $X$, the adjacency matrix $A=A(X)\in \mathcal{M}_n$ is a matrix such that the $(i,j)$-th entry $A_{ij}$ corresponds to the edge weight between vertices $i$ and $j$ (with $A_{ij}=0$ if no such edge exists). If the edge weights between every vertex in the graph are all equal to one, then $X$ is said to be an unweighted graph. We will consider only undirected graphs herein, so that $A_{ij}=A_{ji}$ for all $i,j=1, \dots, n$. The degree matrix $D(X)\in \mathcal{M}_n$ is the diagonal matrix whose components $d_{ii}=\operatorname{deg}(i)$, referred to as the degree of the vertex $i$, correspond to the sum of all weights incident on the the vertex $i$. If the degree of every vertex is equal, then $X$ is said to be regular. The Laplacian $L(X)$ of a  weighted graph $X$ is defined as $L(X)=D(X)-A(X)$. 

    A graph $X$ is Hadamard diagonalizable if its Laplacian is diagonalizable by a Hadamard matrix. Similarly, a graph $X$ is weakly Hadamard  diagonalizable (WHD) if its Laplacian is diagonalizable by a weak Hadamard matrix. Note that if a graph is Hadamard diagonalizable or WHD, then it is diagonalizable by a normalized Hadamard or normalized weak Hadamard, respectively. For a WHD graph $X$, we denote by $P_X$ the normalized weak Hadamard diagonalizing the Laplacian of $X$.
    
    In this section, we consider the combinatorial and spectral properties of weakly Hadamard diagonalizable graphs, and provide constructions and examples of graphs that are weakly Hadamard diagonalizable.

\subsection{Eigenvalues and eigenvectors of WHD graphs}

We say that $X$ is \textit{Laplacian integral} if the Laplacian eigenvalues of $X$ are all integers. The following is an extension of \cite[Lemma 2.2]{adm2021} to weighted WHD graphs. We include the proof for completeness.

\begin{theorem}
\label{lapint}
If $X$ is an integer-weighted WHD graph, then $X$ is Laplacian integral.
\end{theorem}

\begin{proof}
Let $S=[\mathbf{1},\mathbf{x}_2,\ldots,\mathbf{x}_n]$ be the weak Hadamard diagonalizing $L(X)$. Since $L(X)$ has integer entries and each $\mathbf{x}_j$ has entries from $\{0,-1,1\}$, the number $\lambda_j$ satisfying $L(x)\mathbf{x}_j=\lambda_j\mathbf{x}_j$ must be an integer.
\end{proof}

For our next result, we denote the edge weight between vertices $u$ and $v$ by $A_{uv}$. Recall that vertices $u$ and $v$ in $X$ are \textit{twins} if they have the same neighbours. Twins may or may not be adjacent.

\begin{theorem}
\label{nk}
Let $X$ be a weighted WHD graph on $n$ vertices. If $1\leq k\leq \frac{n}{2}$ and $\mathbf{x}$ is an eigenvector of $L(X)$ with corresponding eigenvalue $\lambda$ such that $k$ of its entries are equal to 1, then the following hold.
\begin{enumerate}
\item If $k=1$, then $\mathbf{x}=\mathbf{e}_u-\mathbf{e}_v$ for some pair of distinct vertices $u$ and $v$ in $X$ and $\lambda=\operatorname{deg}(u)+A_{uv}$. In this case, $u$ and $v$ are twin vertices in $X$.
\item If $k\geq 2$, then $\displaystyle\mathbf{x}=\sum_{u\in U}\mathbf{e}_u-\sum_{v\in W}\mathbf{e}_v$ for some disjoint nonempty subsets $U$ and $W$ of $V(X)$ each of size $k$ such that for any two vertices $u\in U$ and $v\in W$, we have
\begin{equation}
\label{tw}
\sum_{w\notin U\cup W}A_{uw}=\sum_{w\notin U\cup W}A_{vw},
\end{equation}
and for a fixed $u\in U$, we have
\begin{equation*}
\lambda=2\sum_{v\in W}A_{uv}+\sum_{w\notin U\cup W}A_{uw}.    
\end{equation*}
In particular, $\lambda$ is even if and only if one side of (\ref{tw}) is even. Further, if $X$ is unweighted, then $\lambda$ is even if and only if each $u\in U$ has an even number of neighbours in $V(X)\backslash (U\cup W)$.
\end{enumerate}
\end{theorem}


\begin{proof}
By assumption, we may relabel the vertices of $X$ so that $\mathbf{x}=[\mathbf{1},-\mathbf{1},\mathbf{0}]^T$. Since $\mathbf{x}$ is orthogonal to $\mathbf{1}$, Theorem \ref{WHDcombi}(1) implies that $k$ entries of
$\mathbf{x}$ are equal to $-1$. Note that (1) is immediate from \cite[Lemma 2.9]{MonterdeELA}. Now, assume $k\geq 2$. Let $U$ and $W$ resp.\ be the set of vertices corresponding to entries of $\mathbf{x}$ equal to 1 and $-1$. If $X_1$, $X_2$ and $X_3$ are the subgraphs of $X$ formed by $U$, $W$ and $V(X)\backslash (U\cup W)$ resp., then
\begin{equation}
\label{lapeq}
\left[ \begin{array}{ccccc}L(X_1)+R_1&-P_1&-P_2 \\ -P_1&L(X_2)+R_2&-P_3 \\ -P_2^T&-P_3^T&L(X_3)+R_3\end{array}\right]\left[\begin{array}{ccccc}\mathbf{1} \\ -\mathbf{1} \\ \mathbf{0} \end{array} \right]=\lambda\left[\begin{array}{ccccc}\mathbf{1} \\ -\mathbf{1}\\ \mathbf{0} \end{array} \right].
\end{equation}
where $R_j$ is diagonal for $j=1,2,3$, $R_1\mathbf{1}=(P_1+P_2)\mathbf{1}$, $R_2\mathbf{1}=(P_1+P_3)\mathbf{1}$ and $R_3\mathbf{1}=(P_2^T+P_3^T)\mathbf{1}$.
$P_2\mathbf{1}=P_3\mathbf{1}$. This proves (\ref{tw}). 
Since $X$ is integer-weighted, $\lambda$ is even if and only if all entries of $P_2\mathbf{1}$ are even. In particular, if $n=2k$, then $X_3$ is an empty graph. In this case, the matrices $P_2,P_2^T,P_3,P_3^T$ and $L(X_3)+R_3$ are absent in $L(X)$, and so we may regard any vertex of $X_1$ as having 0 neighbours in $X_3$. The same argument then yields $\lambda \mathbf{1}=2P_1\mathbf{1}$, from which it follows that $\lambda$ is even. This proves (2).
\end{proof}

If $\mathbf{x}$ in Theorem \ref{nk}(2) has no zero entries, i.e., $k=\frac{n}{2}$, then $\lambda=2\sum_{v\in W}A_{uv}$ for any $u\in U$, and hence $\lambda$ is even. In particular, if $X$ is an integer weighted Hadamard diagonalizable graph, then $k=\frac{n}{2}$ for any column of $P_X$, and so Theorem \ref{nk} implies that each eigenvalue of $L(X)$ is even. This generalizes the fact that Hadamard diagonalizable graphs have even integers \cite[Theorem 5]{barik2011}.

Determining if a graph is WHD involves careful consideration of the eigenvectors of the Laplacian. In particular, if the eigenspace of one eigenvalue is large, one needs to choose the ``correct'' eigenvectors, if they exist, to produce the weak Hadamard matrix that does the job. Theorems~\ref{lapint} and \ref{nk} provide simple-to-check necessary conditions to be WHD. This allows us to rule out graphs that are not WHD.

We end this subsection with a conjecture based on Corollaries~\ref{lem:OddZero} and~\ref{cor:5}:

\begin{conjecture}\label{conj:odd}
    For any odd dimension, there is no WHD graph whose Laplacian eigenspaces have equal algebraic and geometric multiplicities (that is, there is no weak Hadamard matrix with pairwise orthogonal columns diagonalizing the Laplacian matrix of the graph). 
\end{conjecture}

\subsection{Unions and complements}

The \textit{union} of graphs $X_1,\ldots,X_k$, denoted by $\bigsqcup_{j=1}^k X_j$, is the graph whose vertex set is $\bigsqcup_{j=1}^k V(X_j)$ and edge set $\bigsqcup_{j=1}^k E(X_j)$. We also denote the matrix $M$ with its $j$-th column deleted by $M[j]$.

It is already known that the union of two WHD graphs yields a WHD graph \cite[Lemma 2.3]{adm2021}. We add to this result by determining the weak Hadamard that diagonalizes a union of WHD graphs.


\begin{proposition}
\label{union}
Let $X_1,\ldots X_k$ be weighted graphs on $n_1,\ldots,n_k$ vertices, where each $L(X_j)$ diagonalized by the matrix $S_j$ whose first column is $\mathbf{1}$ and all other columns of $S_j$ are orthogonal to $\mathbf{1}$. Then $X=\bigsqcup_{j=1}^k X_j$ is diagonalized by the matrix
\begin{equation}
\label{union1}
[\ \mathbf{1}\ \mathbf{v}_1\ \ldots\ \mathbf{v}_{k-1}\ |\ \bigoplus_{j=1}^k S_j[1]\ ]
\end{equation}
where each $\mathbf{v}_j$ is a vector of order $\sum_{j=1}^kn_j$ given by
\begin{equation}
\label{vj}
\mathbf{v}_j=\mathbf{e}_{j}\otimes  \mathbf{1}_{n_j}-\mathbf{e}_{j+1}\otimes \mathbf{1}_{n_{j+1}}.
\end{equation}
In particular, if each $X_j$ is weighted WHD, and either $k=2$ or $n_1=n_2=\cdots=n_k$, then $X$ is also WHD.
\end{proposition}

\begin{proof}
Since $S_j$ has $\mathbf{1}$ as its first column and all other columns are orthogonal to $\mathbf{1}$, the vectors $\mathbf{e}_j\otimes\mathbf{1}_{n_k}$ for each $j\in\{1,\ldots,k\}$ are eigenvectors associated with the eigenvalue 0 of $L\left(X\right)$. Observe that $B=\{\mathbf{1},\mathbf{v}_1,\ldots,\mathbf{v}_{k-1}\}$ is a linearly independent set of eigenvectors associated with the eigenvalue 0 of $L\left(X\right)$, and so the matrix in (\ref{union1}) diagonalizes $X$. Moreover, each vector in $B$ has entries from the set $\{0,-1,1\}$ and is orthogonal to every column of $\bigoplus_{j=1}^k S_j[1]$. Thus, if each $X_j$ is weighted WHD, and either $k=2$ or $n_1=\ldots=n_k$, then the non-consecutive elements in $B$ are orthogonal, and so $X$ is WHD.
\end{proof}

\begin{remark}
\label{D}
Suppose one of the $X_j$'s in Proposition \ref{union} is $K_2\sqcup O_1$. Then $D=\begin{bmatrix} 1&0&1 \\ 1&0&-1 \\ 0&1&0 \end{bmatrix}$ diagonalizes $L(K_2\sqcup O_1)$. Note that the first two columns of $D$ are orthogonal eigenvectors associated with the eigenvalue 0. Thus, if $\mathbf{v}$ is a vector such that $\{\mathbf{1},\mathbf{v}\}$ spans the eigenspace of $L(K_2\sqcup O_1)$ associated with 0, then $\mathbf{v}$ cannot be orthogonal to $\mathbf{1}$. In this case, there is no matrix $S_j$ whose first column is $\mathbf{1}$ and all other columns of $S_j$ are orthogonal to $\mathbf{1}$. One also checks that the vector $\mathbf{v}_1$ in (\ref{vj}) is not orthogonal to $\mathbf{1}$, and so the columns of the matrix in (\ref{union1}) are not pairwise orthogonal. That is, the conclusion of Proposition \ref{union} does not necessarily hold whenever the premise is not true.
\end{remark}

If $k=2$, then the last statement of Proposition \ref{union} holds for Hadamard diagonalizable graphs whenever $X=Y$ \cite[Lemma 7]{barik2011}, but could fail whenever $X\neq Y$. Indeed, Breen et al.\ showed that the only disconnected graphs on $8k+4$ vertices that are Hadamard diagonalizable are $K_{2k+2}\sqcup K_{2k+2}$ and $O_{8k+4}$ \cite[Theorem 5]{breen}. Thus, if $X$ and $Y$ are Hadamard diagonalizable, and either $X\neq K_{2k+2}$ or $X$ is not an empty graph, then $X\sqcup Y$ is not Hadamard diagonalizable whenever $X\sqcup Y$ has $8k+4$ vertices.

\begin{corollary}
\label{power2}
With the assumption in Proposition \ref{union}, suppose each $X_j$ is weighted WHD and $n_1=\ldots=n_k$. If $k=2^{\ell}$ and each $S_j$ has pairwise orthogonal columns, then $X$ is WHD and $L(X)$ is diagonalized by 
\begin{equation*}
[\ Q\ |\ \bigoplus_{j=1}^k S_j[1]\ ]
\end{equation*}
with pairwise orthogonal columns, where $Q=\begin{bmatrix} 1&1\\ 1&-1\end{bmatrix} \otimes \mathbf{1}_{n_1}$ if $\ell=1$ and $P_{\ell-1}\otimes \mathbf{1}_{n_1}$ otherwise, where $P_\ell$ is the matrix in (\ref{Pell}). 

\end{corollary}

\begin{proof}
From the proof of Proposition \ref{union}, the vectors $\mathbf{e}_j\otimes \mathbf{1}_{n_k}$ for each $j\in\{1,\ldots,k\}$ are eigenvectors associated with the eigenvalue 0 of $L\left(X\right)$. If $\ell=1$, then we may choose $v_1$ such that the first two columns of the matrix in (\ref{union1}) is $\begin{bmatrix} 1&1\\ 1&-1\end{bmatrix} \otimes \mathbf{1}_{n_1}$. For $\ell\geq 2$, we may choose the $\mathbf{v}_j$'s such that first $k$ columns of the matrix in (\ref{union1}) is $P_{\ell-1}\otimes \mathbf{1}_{n_1}$, where $P_\ell$ is the matrix in (\ref{Pell}). In both cases, the matrix in (\ref{union1}) has pairwise orthogonal columns. Since $S_j$ has pairwise orthogonal columns, the result is immediate.
\end{proof}

It is known that the complement $X^c$ and the join $X\vee X$ are Hadamard diagonalizable whenever $X$ is \cite[Lemma 7]{barik2011}. For $X^c$, this is known to extend to WHD graphs under mild conditions \cite[Lemmas 2.4-2.5]{adm2021}.

\begin{proposition}
\label{comp}
Let $X$ be unweighted. If $L(X)$ is diagonalized by a matrix $S$  each of whose columns distinct from $\mathbf{1}$ is orthogonal to $\mathbf{1}$, then $X^c$ is also diagonalized by $S$. In particular, if $X$ is WHD, then so is $X^c$.
\end{proposition}

The assumption about the matrix $S$ in Proposition \ref{comp} is indeed necessary. For instance, $X=K_2\sqcup O_1$ is diagonalized by the weak Hadamard matrix $D$ is Remark \ref{D}. But one checks that $L(P_3)$ has 3 as a simple eigenvalue with $[1,-2,1]^T$ as an associated eigenvector, and so $X^c=P_3$ is not WHD. Thus, the conclusion of Proposition \ref{comp} does not necessarily hold if the premise is not true.

If $X$ is connected, then all eigenvectors of $L(X)$ that are not in the eigenspace associated with the eigenvalue 0 are orthogonal to $\mathbf{1}$, and so  $X^c$ is WHD by Proposition \ref{comp}. On the other hand, if $X=\bigsqcup_{j=1}^kX_j$, where the $X_j$'s are WHD graphs on the same number of vertices and each $L(X_j)$ diagonalized by the matrix $S_j$ whose columns other than $\mathbf{1}$ are orthogonal to $\mathbf{1}$, then $X^c$ is also WHD. In particular, if $X$ is a disconnected regular graph whose components are WHD and have equal sizes, then $X^c$ is also WHD by the preceding statement \cite[Lemma 2.5]{adm2021}. But as the next example shows, there are disconnected non-regular WHD graphs whose components have equal size, where the complement also happens to be WHD.

\begin{example}
\label{uncomp}
Amongst all unweighted graphs on four vertices, four are non-isomorphic Hadamard diagonalizable, namely $K_4$, $C_4$, $K_2\sqcup K_2$ and $O_4$. Moreover, there are two  non-isomorphic WHD graphs on four vertices that are not Hadamard diagonalizable with $P_X$ having pairwise orthogonal columns, namely $K_4\backslash e$ and $O_2\sqcup K_2$. In fact, the Laplacian matrices of these six graphs are diagonalizable by the matrix $P_1$ in (\ref{Pone}), which is a weak Hadamard with pairwise orthogonal columns. Let $k=2^{\ell}$ and $X_{(k)}=\bigsqcup_{j=1}^kX_j$, where $X_j\in\{K_4,C_4,K_2\sqcup K_2,O_4,K_4\backslash e,O_2\sqcup K_2\}$. 
\begin{enumerate}
\item By Corollary \ref{power2}, $X_{(k)}$ is WHD and $P_{X_{(k)}}=[Q\ |\ \bigoplus_{j=1}^k P_1[1]\ ]$, where $Q=\begin{bmatrix} 1&1\\ 1&-1\end{bmatrix} \otimes \mathbf{1}_{n_1}$ if $\ell=1$ and $P_{\ell-1}\otimes \mathbf{1}_{n_1}$ otherwise, where $P_\ell$ is the matrix in (\ref{Pell}). 
\item Suppose at least two of the $X_j$'s are distinct so that $X_{(k)}$ is not regular. By  Proposition \ref{comp}, $X_{(k)}^c$ is WHD, and so $\mathcal{F}=\{X_{(k)}:k=2^{\ell},\ell\geq 1\}$ is an infinite family of disconnected non-regular WHD graphs, where the $X_j$'s have equal sizes and each $X_k^c$ is WHD. If we restrict each $X_j\in\{K_4,C_4,K_4\backslash e\}$, then each $X_{(k)}\in \mathcal{F}$ have components of equal sizes. Moreover, neither $X_{(k)}$ nor $X_{(k)}^c$ in this case are Hadamard diagonalizable, but $P_{X_{(k)}}=P_{X_{(k)}^c}$ has pairwise orthogonal columns.
\end{enumerate}
\end{example}

\subsection{Joins}

The \textit{join} of two weighted graphs $X$ and $Y$, denoted $X\vee Y$, is the graph obtained from $X$ and $Y$ by adding all edges between them of weight one \cite{kirkland2023quantum}. We show that under mild conditions, $X\vee X$ is WHD whenever $X$ is. It is a special case of \cite[Lemma 4.3]{adm2021}, however, the proof herein does not rely on recursively balanced partitions, and it explicitly shows the weak Hadamard matrix that diagonalizes $X\vee X$. 

\begin{theorem}\label{thm:join}
Let $X$ and $Y$ be weighted graphs, and $S$ be a matrix whose first column is $\mathbf{1}$ and all other columns are orthogonal to $\mathbf{1}$. If $L(X)$ and $L(Y)$ are diagonalized by $S$, then $L(X\sqcup Y)$ and $L(X\vee Y)$ are diagonalized by $\begin{bmatrix} S&S \\ S&-S\end{bmatrix}$ if and only if $X=Y$. In particular, if $\Lambda=S^{-1}L(X)S$, where $\Lambda=\text{diag}(0,\lambda_2,\ldots,\lambda_n)$, then
\begin{equation*}
\Lambda'=\frac{1}{2}\left[ \begin{array}{ccccc} S^{-1} &S^{-1} \\  S^{-1}&-S^{-1}\end{array}\right]L(X\vee X)\left[ \begin{array}{ccccc} S &S \\  S&-S\end{array}\right],
\end{equation*}
where $\Lambda'=\text{diag}(0,\lambda_2+n,\ldots,\lambda_n+n,2n,\lambda_2+n,\ldots,\lambda_n+n)$. Further, if $X$ is WHD, then so are $X\sqcup X$ and $X\vee X$.
\end{theorem}

\begin{proof}
Let $S=[\mathbf{1},\ldots,\mathbf{x}_n]$ such that $\Lambda_X=S^{-1}L(X)S$ and $\Lambda_Y=S^{-1}L(Y)S$, where $\Lambda_X$ and $\Lambda_Y$ are diagonal matrices. Since $L(X\vee Y)= \begin{bmatrix}
L(X)+nI&-\mathbf{J} \\  -\mathbf{J}&L(Y)+nI
\end{bmatrix}$, where $\mathbf{J}$ is the square all-ones matrix of appropriate size, one checks that 
\begin{align*}
\Lambda'&=\frac{1}{2}\begin{bmatrix} S^{-1}&S^{-1} \\ S^{-1}&-S^{-1}\end{bmatrix}L(X\vee Y)\begin{bmatrix} S&S \\ S&-S\end{bmatrix}\\
&=\begin{bmatrix} \frac{1}{2}(\Lambda_X+\Lambda_Y) +nI -S^{-1}\mathbf{J}S&\Lambda_X-\Lambda_Y \\  \Lambda_X-\Lambda_Y&\frac{1}{2}(\Lambda_X+\Lambda_Y) +nI +S^{-1}\mathbf{J}S\end{bmatrix}
\end{align*}
Note that the $(u,v)$ entry of $S^{-1}\mathbf{J}S$ is given by $\mathbf{e}_u^TS^{-1}\mathbf{J}S\mathbf{e}_v=\mathbf{e}_u^TS^{-1}\mathbf{J}\mathbf{x}_v$. By assumption of the columns of $S$ being orthogonal to $\mathbf{1}$, we have $\mathbf{J}\mathbf{x}_v=0$ unless $v=1$. Now, note that $S^{-1}$ has first row equal to $\frac{1}{n}\mathbf{1}$, and the rest are orthogonal to $\mathbf{1}$. Thus, if $v = 1$, then we obtain
\begin{align*}
\mathbf{e}_u^TS^{-1}\mathbf{J}\mathbf{x}_1=n\mathbf{e}_u^TS^{-1}\mathbf{1}=n\mathbf{e}_u^T\mathbf{e}_1.
\end{align*}
Consequently, the $(u,v)$ entry of $S^{-1}\mathbf{J}S$ is equal to $n$ whenever $u=v=1$ and 0 otherwise. This implies that $\Lambda'$ is diagonal if and only if $\Lambda_X=\Lambda_Y$, i.e., $X=Y$. In particular, if $X=Y$, then one checks that $\Lambda'=(0,\lambda_2+n,\ldots,\lambda_n+n,2n,\lambda_2+n,\ldots,\lambda_n+n)$. The same argument applies to $X\sqcup Y$.
\end{proof}

We note that Theorem \ref{thm:join} applies to both connected and disconnected graphs as long as the hypothesis about the matrix $S$ is satisfied. Moreover, we already know from Proposition \ref{comp} that $X\sqcup Y$ is always WHD whenever $X$ and $Y$ are. Now, notice that the matrix $\begin{bmatrix} \mathbf{1}&S[1]&\mathbf{1}&0 \\ \mathbf{1}&0&-\mathbf{1}&S[1] \end{bmatrix}$ in (\ref{union1}) diagonalizing $L(X\sqcup Y)$ can only be transformed to $\begin{bmatrix} S&S \\ S&-S\end{bmatrix}$ via column operations if and only if the eigenvalues of $L(X)$ and $L(Y)$ having the $j$th column of $S[1]$ as their eigenvector are equal for each $j$. Equivalently, $X=Y$. This shows that $\begin{bmatrix} S&S \\ S&-S\end{bmatrix}$ cannot diagonalize $L(X\sqcup Y)$ whenever $X\neq Y$.

For any integer $k\geq 2$, let $Z_k=X\vee\ldots \vee X$ denoted the $k$-fold join of $X$ with itself. More generally, if $X$ is a connected weighted WHD, then for all integers $k\geq 2$, $Z_k$ is also weighted WHD by \cite[Lemma 4.3]{adm2021}. The next result follows from Theorem \ref{thm:join} by induction.

\begin{corollary}
\label{tj}
With the assumption in Theorem \ref{thm:join}, we further suppose that $X$ is WHD and $P_X$ has pairwise orthogonal columns. Let $T_1=P_X$ and for $\ell\geq 2$, suppose $T_{\ell}=\left[ \begin{array}{ccccc} T_{\ell-1} &T_{\ell-1} \\  T_{\ell-1}&-T_{\ell-1}\end{array}\right]$. If $k=2^{\ell}$, then $Z_k$ is WHD and $P_{Z_k}=T_{\ell+1}$ has pairwise orthogonal columns.
\end{corollary}

\begin{example}
\label{n=4}
Let $k=2^{\ell}$ and suppose $X\in\{K_4,C_4,K_2\sqcup K_2,O_4,K_4\backslash e,O_2\sqcup K_2\}$. Set $T_1=P_1$. As we know, $P_X=T_1$ from Example \ref{uncomp}. Invoking Corollary \ref{tj}, each $Z_k$ is WHD and $P_{Z_k}=T_{\ell+1}$ has pairwise orthogonal columns. In particular, if $X\in\{K_4\backslash e,O_2\sqcup K_2\}$, then $\{Z_k: 2^k, \ell\geq 1\}$ is an infinite family of WHD graphs that are not Hadamard diagonalizable, but $P_{Z_k}$ has pairwise orthogonal columns.
\end{example}

\subsection{Merge}

Let $X$ and $Y$ be two weighted graphs, each with $n$ vertices. The \textit{merge} of $X$ and $Y$ with respect to the integer weights $w_1$ and $w_2$, denoted $X \tensor[_{w_1}]{\odot}{_{w_2}} Y$, is the graph with Laplacian matrix
\begin{eqnarray}
\label{lapmerge}
    \begin{bmatrix} w_1L(X)+w_2D(Y)&-w_2A(Y) \\ -w_2A(Y)&w_1L(X)+w_2D(Y)\end{bmatrix}.
\end{eqnarray}
If $w_1=w_2=1$, then we simply write $X \tensor[_{w_1}]{\odot}{_{w_2}} Y$ as $X \odot Y$. In this case, if $X$ and $Y$ are graphs on the same vertex set that do not have an edge in common, then $X \odot Y$ is a \textit{double cover} of a graph with Laplacian matrix $L(X)+L(Y)$. In particular, if $X$ is unweighted and $Y=X^c$, then $X \odot Y$ is a double cover of $K_n$ (also called the \textit{switching graph} of $X$), while if $X=O_n$, then $X \odot Y$ is called the \textit{bipartite double} of $Y$ (also called the \textit{canonical double cover} of $Y$). The merge operation was introduced by Johnston et al., and was used to produce Hadamard diagonalizable graphs from smaller ones \cite{johnston2017}. As our next result implies, the same goal is achieved by the merge operation for WHD graphs.

\begin{theorem}\label{thm:merge}
Let $X$ be a weighted graph and $S$ be a matrix whose first column is $\mathbf{1}$ and all other columns are orthogonal to $\mathbf{1}$. Suppose $L(X)$ and $L(Y)$ are diagonalized by the same matrix $S$. Then $L(X \tensor[_{w_1}]{\odot}{_{w_2}} Y)$ is diagonalized by $\begin{bmatrix} S&S \\ S&-S\end{bmatrix}$ if and only if $Y$ is a weighted-regular graph. In particular, if $Y$ is weighted $k$-regular and we let $\Lambda_1=S^{-1}L(X)S$ and $\Lambda_2=S^{-1}L(Y)S$, where $\Lambda_1=\text{diag}(0,\lambda_2,\ldots,\lambda_n)$ and $\Lambda_2=\text{diag}(0,\theta_2,\ldots,\theta_n)$, then
\begin{equation*}
\Lambda'=\frac{1}{2}\left[ \begin{array}{ccccc} S^{-1} &S^{-1} \\  S^{-1}&-S^{-1}\end{array}\right]L(X \tensor[_{w_1}]{\odot}{_{w_2}} Y)\left[ \begin{array}{ccccc} S &S \\  S&-S\end{array}\right],
\end{equation*}
where $ \hspace{-0.02in} \Lambda'\hspace{-0.03in} =\hspace{-0.02in} \text{diag}(0,w_1\lambda_2+w_2\theta_2,\ldots,w_1\lambda_n+w_2\theta_n,2w_2k,w_1\lambda_2+w_2(2k-\theta_2),\ldots,w_1\lambda_n+w_2(2k-\theta_n)$.
\end{theorem}

\begin{proof}
     Using (\ref{lapmerge}), one checks that 
\begin{align*}
\Lambda'&=\frac{1}{2}\begin{bmatrix} S^{-1}&S^{-1} \\ S^{-1}&-S^{-1}\end{bmatrix}L(X \tensor[_{w_1}]{\odot}{_{w_2}} Y)\begin{bmatrix} S&S \\ S&-S\end{bmatrix}\\
&=\begin{bmatrix} w_1\Lambda_1+w_2\Lambda_2 &0 \\  0&w_1\Lambda_1 -w_2\Lambda_2+2w_2S^{-1}D(Y)S\end{bmatrix}.
\end{align*}
Thus, $\Lambda'$ is diagonal if and only if $S^{-1}D(Y)S$ is diagonal. However, $S^{-1}D(Y)S=F$ for some diagonal matrix $F$ if and only if $D(Y)S=SF$. If we let $D(Y)=\text{diag}(d_1,\ldots,d_n)$ and $F=\text{diag}(f_1,\ldots,f_n)$, then $D(Y)S=SF$ if and only if $d_iS_{ij}=f_jS_{ij}$ for each $i$ and $j$. Equivalently, $d_i=f_j$ for each $i$ and $j$, i.e., $D(Y)$ is a scalar multiple of the identity. The rest is straightforward.
\end{proof}

The following is immediate from Theorem \ref{thm:merge}.

\begin{corollary}
\label{merge}
Let $X$ and $Y$ be weighted WHD graphs such that $P_X=P_Y$. Then $Z=X \tensor[_{w_1}]{\odot}{_{w_2}} Y$ is weighted WHD if and only if $Y$ is weighted-regular, in which case $P_Z=\begin{bmatrix} P_X&P_X \\ P_X&-P_X\end{bmatrix}$. The following also hold.
\begin{enumerate}
\item If we add that $Y$ is weighted Hadamard diagonalizable, then $X \tensor[_{w_1}]{\odot}{_{w_2}} Y$ is weighted WHD.
\item If $X$ is weighted-regular, then $X \tensor[_{w_1}]{\odot}{_{w_2}} X$ is weighted WHD. Moreover, if $X$ is unweighted and regular, then $X \tensor[_{w_1}]{\odot}{_{w_2}} X^c$ and $X \odot X^c$ are weighted WHD graphs.
\end{enumerate}
\end{corollary}

Since $\begin{bmatrix} S&S \\ S&-S\end{bmatrix}$ in Theorem \ref{thm:merge} is a Hadamard matrix whenever $S$ is, we see that $X \tensor[_{w_1}]{\odot}{_{w_2}} Y$ is Hadamard diagonalizable whenever $X$ and $Y$ are Hadamard diagonalizable. This was first observed in \cite{johnston2017}, and is generalized to WHD graphs by Corollary \ref{merge}.

Let $X$ and $Y$ be weighted WHD graphs with $P_X=P_Y$. Unlike $X\vee Y$, it is possible for $X \tensor[_{w_1}]{\odot}{_{w_2}} Y$ to be diagonalized by $\begin{bmatrix} P_X&P_X \\ P_X&-P_X\end{bmatrix}$ even if $X\neq Y$. Thus, the merge operation is advantageous in producing bigger weighted WHD graphs from smaller ones.


\begin{example}
Let $X\in\{K_4\backslash e,O_2\sqcup K_2,K_4,C_4,K_2\sqcup K_2,O_4\}$ and $Y\in\{K_4,C_4,K_2\sqcup K_2,O_4\}$. From Example \ref{n=4}, $X$ is WHD, $Y$ is Hadamard diagonalizable (and thus, regular), and the Laplacian matrices of both $X$ and $Y$ are diagonalized by the matrix $P_1$ in (\ref{Pone}), which has pairwise orthogonal columns. By Corollary \ref{merge}(1), we conclude that $Z=X \tensor[_{w_1}]{\odot}{_{w_2}} Y$ is a weighted WHD graph and $P_Z=\begin{bmatrix} P_1&P_1 \\ P_1&-P_1\end{bmatrix}$.
\end{example}

For our next example, recall that conference graphs are strongly regular graphs whose number of vertices must be congruent to 1 (mod 4). In \cite[Theorem 5.9]{adm2021}, an infinite family of conference graphs are shown to be WHD.

\begin{example}
\label{confe}
Suppose $X$ is isomorphic to either (i) $K_n$ with $n\not\equiv 0$ (mod 4), (ii) $K_{2n}$ minus a perfect matching, where $n$ is odd, (iii) $K_{n,\ldots,n}$ with $mn\not\equiv 0$ (mod 4) number of vertices, or (iv) a conference graph that is WHD. Then the number of vertices of $X$ is not a multiple of $4$, and so $X$ is not Hadamard diagonalizable. However, the graphs in (i), (ii) and (iii) are WHD by Lemma 1.5, Corollary 4.4 and Corollary 4.9 in \cite{adm2021}, respectively. Thus, $X\vee X$ is WHD by Theorem \ref{thm:join}. Moreover, since $X$ is unweighted and regular, Corollary \ref{merge}(2) implies that $X \tensor[_{w_1}]{\odot}{_{w_2}} X$ and $X \tensor[_{w_1}]{\odot}{_{w_2}} X^c$ are weighted WHD graphs. Further, the weak Hadamard diagonalizing $L(Z)$, where $Z\in\{X\vee X,X \tensor[_{w_1}]{\odot}{_{w_2}} X,X \tensor[_{w_1}]{\odot}{_{w_2}} X^c\}$, is $\begin{bmatrix} P_X&P_X \\ P_X&-P_X\end{bmatrix}$. This yields infinite families of WHD graphs that are not Hadamard diagonalizable and the columns of $P_X$ are not pairwise orthogonal.
\end{example}

It is also worth mentioning that if $X$ is the Kneser graph $K(5,2)$ (the Petersen graph) or $K(6,2)$, then \cite[Lemma 5.14]{adm2021} and the same argument used in Example \ref{confe} imply that $X\vee X$, $X \tensor[_{w_1}]{\odot}{_{w_2}} X$ and $X \tensor[_{w_1}]{\odot}{_{w_2}} X^c$ are weighted WHD graphs that are not Hadamard diagonalizable.


\begin{corollary}
\label{bipdoub}
Let $Y$ be a weighted-regular WHD graph. Then the bipartite double of $Y$ is WHD.
\end{corollary}

\begin{proof}
The bipartite double of a bipartite graph $Y$ is simply $Y\sqcup Y$. Thus, if $Y$ is bipartite, then the bipartite double of $Y$ is WHD by Proposition \ref{union}. If $Y$ is non-bipartite, then taking $w_1=w_2=1$ and $X$ as the empty graph in Corollary \ref{merge} yields the desired result.
\end{proof}

\begin{example}
If $Y$ is one of the graphs in (i)-(iv) in Example \ref{confe} with the additional condition that $n\geq 3$ in (i) and $m\geq 3$ in (iii), then $Y$ is a regular non-bipartite weighted graph that is WHD, and so Corollary \ref{bipdoub} implies that the bipartite double of $Y$ is a connected WHD graph that is not Hadamard diagonalizable. On the other hand, if $Y\in\{C_6,K_{n,n}\}$, then $Y$ is regular, bipartite and WHD, and so the bipartite double of $Y$ is a disconnected WHD graph by Corollary \ref{bipdoub}.
\end{example}

For two weighted graphs $X$ and $Y$, denote by $X\square Y$, $X\boxtimes Y$ and $X\times Y$ the Cartesian, strong and direct products of $X$ and $Y$, which are graphs whose adjacency matrices are given by $A(X)\otimes I+I\otimes A(Y)$, $A(X)\otimes A(Y)$ and $A(X)\otimes I+I\otimes A(Y)+A(X)\otimes A(Y)$ respectively. Part (1) of the following result can be viewed as an extension of \cite[Proposition 3.3]{adm2021}.

\begin{theorem}
\label{prod}
Let $X$ and $Y$ be weighted WHD graphs.
\begin{enumerate}
\item  Suppose $X$ and $Y$ are weighted regular, and let $\star\in\{\square,\boxtimes,\times\}$. If $P_X$ has pairwise orthogonal columns, then $Z=X\star Y$ is WHD with $P_{Z}=P_X\otimes P_Y$.
\item If $P_X=P_Y$, then $Z=X\tensor[_{w_1}]{\odot}{_{w_2}}Y$ is WHD with $P_Z=\begin{bmatrix} P_X&P_X \\ P_X&-P_X\end{bmatrix}$.
\end{enumerate}
If we add that $P_Y$ has pairwise orthogonal columns, then so does $P_{Z}$ in both cases.
\end{theorem}

\begin{proof}
(1) is immediate from Proposition \ref{prop:tpAB} and the fact that $P_X\otimes P_Y$ diagonalizes $X\star Y$. (2) follows from Corollary \ref{merge}(1).
\end{proof}




We end this section with some examples of connected WHD graphs on eight vertices where the Laplacian matrix is diagonalized by a weak Hadamard with pairwise orthogonal columns.

\begin{example}
\label{eight}
Let $X,Y\in\{K_4,C_4,K_2\sqcup K_2,O_4,K_4\backslash e,O_2\sqcup K_2\}$. Consider $R=\begin{bmatrix} \mathbf{1}&\mathbf{1}&P_1[1]&0 \\ \mathbf{1}&-\mathbf{1}&0&P_1[1]\end{bmatrix}$ and $T=\begin{bmatrix} P_1&P_1 \\ P_1&-P_1\end{bmatrix}$, where $P_1$ is given in (\ref{Pone}). Note that $R$ and $S$ are weak Hadamards with pairwise orthogonal columns. We note that  $R$ is equivalent (via column permutations) to the weak Hadamard matrix of order $2^3$ whose construction is outlined in Section~\ref{sec:constr}(6). 

The following connected unweighted graphs on 8 vertices have their Laplacian matrices diagonalized by a weak Hadamard with pairwise orthogonal columns.
\begin{enumerate}
\item $(X\sqcup Y)^c$, diagonalized by $R$
\item $X\vee X$, diagonalized by $T$.
\item $X \odot Y$ with $X\neq O_4$ and $Y\in\{K_4,C_4,K_2\sqcup K_2\}$, diagonalized by $T$. 
\item The bipartite double of $K_4$ (isomorphic to $K_{4,4}$ minus a perfect matching), diagonalized by $T$.
\end{enumerate}
\end{example}

\begin{remark}
Since the set $\{K_4,C_4,K_2\sqcup K_2,O_4,K_4\backslash e,O_2\sqcup K_2\}$ is closed under complements and $X\vee Y=(X^c\sqcup Y^c)^c$ for unweighted graphs $X$ and $Y$, Example \ref{eight}(1) yields 21 graphs of the form $X\vee Y$,where  $X,Y\in\{K_4,C_4,K_2\sqcup K_2,O_4,K_4\backslash e,O_2\sqcup K_2\}$. This includes the 6 graphs in Example \ref{eight}(2). As $K_4\backslash e\vee K_4\backslash e$ and $C_4\vee K_4$ are isomorphic, 20 amongst these 21 graphs are nonisomorphic.

Since $X \odot Y$ and $Y \odot X$ are isomorphic, Example \ref{eight}(3) yields 8 nonisomorphic connected graphs, namely $(O_2\sqcup K_2) \odot K_4$, $(K_2\sqcup K_2) \odot C_4$ (isomorphic to $O_4\odot K_4$ in Example \ref{eight}(4), the hypercube on 8 vertices), $(K_2\sqcup K_2) \odot K_4$ (isomorphic to $K_2\square K_4$, the complement of the hypercube on 8 vertices), $K_4 \odot K_4$ (isomorphic to $C_4\vee C_4$), $C_4 \odot K_4$ (isomorphic to $(K_2\sqcup K_2)\vee (K_2\sqcup K_2)$), $C_4 \odot C_4$ (isomorphic to $O_4\vee O_4$), $K_4\backslash e \odot K_4$ (isomorphic to $(K_2\sqcup K_2)\vee C_4$), $K_4\backslash e\odot  C_4$ (isomorphic to $(K_2\sqcup K_2)\vee C_4$).

Consequently, there are at least 23 nonisomorphic connected WHD graphs on 8 vertices whose Laplacian matrices are diagonalized by a weak Hadamard matrix with pairwise orthogonal columns. We display them in Table \ref{tab}. Note that 6 graphs of these 23 graphs are Hadamard diagonalizable (see \cite[Table 1]{breen}), namely $K_4\vee K_4\cong K_8$, $ C_4\vee C_4\cong K_{2,2,2,2}$, $(K_2\sqcup K_2)\vee (K_2\sqcup K_2)$, $O_4\vee O_4\cong K_{4,4}$, $O_2\odot K_4$ (the hypercube on 8 vertices), and $K_2\square K_4\cong (O_2\odot K_4)^c$. The other 17 graphs are not Hadamard diagonalizable because they are not regular. Further, the three graphs $(O_2\sqcup K_2) \odot K_4$, $O_4\odot K_4$, and $K_2\square K_4$ are diagonalized by the matrix $T$ in Example \ref{eight}, while the rest are diagonalized by the matrix $R$ in Example \ref{eight}. This is done through a careful reordering of the Laplacian spectrum of each graph. 
\end{remark}

\section{State transfer in weakly Hadamard diagonalizable graphs}\label{sec:qst}

Motivated by quantum information theory, we are interested in whether a given graph $X$ exhibits different types of quantum state transfer. 
The transition matrix of the graph $X$ is the time-dependent unitary matrix $U(t)=e^{itL(X)}$. We say that perfect state transfer (PST) occurs if there exists a time $t$ such that a standard basis vector $\mathbf{e}_u$, called the initial state, evolves to a different state $\mathbf{e}_v$ up to a phase factor. More formally, the graph $X$ exhibits PST between vertices $u$ and $v$ if there exists a time $t$ such that
\begin{align}
 U(t)\mathbf{e}_u=\gamma \mathbf{e}_v, 
\end{align}
where $\gamma\in\mathbb{C}$. Since $U(t)$ is unitary, $\gamma$  satisfies $|\gamma|^2=1$, and so PST may be equivalently expressed as
\begin{align}
|\mathbf{e}_u^T U(t) \mathbf{e}_v|^2=1.
\end{align}
If $u=v$ in the above, then vertex $u$ is said to be periodic. Periodicity and strong cospectrality are necessary conditions for PST (see Section 3 and Lemma 14.1 in \cite{Godsil2011StateTO}, respectively).

From Theorem \ref{lapint}, integer-weighted WHD graphs are Laplacian integral. Using a characterization of periodic vertices due to Godsil and Coutinho \cite[Corollary 7.6.2]{Coutinho2021}, we conclude that such graphs are periodic, and are therefore good candidates for PST. This prompts us to characterize perfect state transfer
in WHD graphs under mild conditions. But in order to do this, we first need to characterize strong cospectrality in WHD graphs. Henceforth, we assume all graphs are connected.

\subsection{Strong cospectrality}

Let $X$ be a graph with Laplacian matrix $L=L(X)$. Suppose the spectral decomposition of $L$ is given by $$L=\sum_{j=1}^r\lambda_jE_j,$$ where $E_j$ is the orthogonal projection matrix onto the eigenspace associated with $\lambda_j$, for each $j$. The \textit{eigenvalue support} $\sigma_u$ of a vertex $u$ is the set
\begin{equation*}
\sigma_u=\{\lambda_j:E_j\mathbf{e}_u\neq \mathbf{0}\}.
\end{equation*}
We say that two vertices $u$ and $v$ are \textit{cospectral} if $(E_j)_{u,u}=(E_j)_{v,v}$ for each $j$, and \textit{strongly cospectral} if $E_j\mathbf{e}_u=\pm E_j\mathbf{e}_v$ for each $j$, in which case we define the sets
\begin{equation*}
\sigma_{uv}^+=\{\lambda_j:E_j\mathbf{e}_u=E_j\mathbf{e}_v\}\quad \text{and}\quad \sigma_{uv}^-=\{\lambda_j:E_j\mathbf{e}_u=-E_j\mathbf{e}_v\}.
\end{equation*}
Cospectral vertices have the same eigenvalue supports. Moreover, strongly cospectral vertices are cospectral, but the converse is not true. For more about strong cospectrality, see Godsil and Smith \cite{godsil2024strongly}. 

The following result will be useful in characterizing strong cospectrality in WHD graphs. For each eigenvalue $\lambda$ of $L$, we let $W(\lambda)$ be an orthogonal basis of eigenvectors associated with $\lambda$. Some simple Python code that checks strong cospectrality according to Lemma~\ref{SC}  is available for download at \cite{sup}. 

\begin{lemma}
\label{SC}
Vertices $u$ and $v$ are strongly cospectral if and only if for each $\lambda\in\sigma_u$, either $\mathbf{x}(u)=\mathbf{x}(v)$ for all $\mathbf{x}\in W(\lambda)$ or $\mathbf{x}(u)=-\mathbf{x}(v)$ for all $\mathbf{x}\in W(\lambda)$. Moreover, if $u$ and $v$ are strongly cospectral, then
\begin{equation*}
\hspace{-0.1in} \sigma_{uv}^+=\{\lambda:\mathbf{x}(u)=\mathbf{x}(v)\neq 0\ \text{for all}\ \mathbf{x}\in W_{\lambda}\}\quad \text{and}\quad \sigma_{uv}^-=\{\lambda:\mathbf{x}(u)=-\mathbf{x}(v)\neq 0\ \text{for all}\ \mathbf{x}\in W_{\lambda}\}.
\end{equation*}
\end{lemma}  

Since $W_\lambda$ forms a basis for the eigenspace corresponding to the eigenvalue $\lambda$, the property that $\mathbf{x}(u)=\mathbf{x}(v)$ (similarly $\mathbf{x}(u)=-\mathbf{x}(v)$) for all $\mathbf{x}\in W_\lambda$ extends to every eigenvector in the eigenspace of $\lambda$. Hence while the preceding lemma requires analysing eigenvectors which form orthogonal bases to prove strong cospectrality, in principle, one could rule out strong cospectrality between two vertices $u$ and $v$ by finding a single eigenvector that does not have its $u$ and $v$-th components being equal up to a sign, or by finding two distinct eigenvectors associated with the same eigenvalue where the $u$- and $v$-th components are equal in one but are opposite in signs in the other.

The following result characterizes (strong) cospectrality in WHD graphs where the weak Hadamard has pairwise orthogonal columns. 

\begin{lemma}
\label{WHDSC}
Let $X$ be a weighted WHD graph For every eigenvalue $\lambda$ of $L(X)$, let $W(\lambda)$ be the set of columns of $P_X$ that are eigenvectors corresponding to $\lambda$.
\begin{enumerate}
\item The eigenvalue support of $u$ is $\sigma_u=\{\lambda:\mathbf{x}(u)\neq 0\ \text{for some}\ \mathbf{x}\in W(\lambda)\}$.
\item Assume $P_X$ has pairwise orthogonal columns. Vertices $u$ and $v$ are cospectral if and only if for every $\lambda\in\sigma_u$, either $W_u(\lambda)=W_v(\lambda)$ or
\begin{equation}
\label{cosp1}
\sum_{\mathbf{x}\in W_v(\lambda)\backslash W_u(\lambda)}\frac{1}{\|\mathbf{x}\|^2}=\sum_{\mathbf{x}\in W_u(\lambda)\backslash W_v(\lambda)}\frac{1}{\|\mathbf{x}\|^2},
\end{equation}
where $W_u(\lambda)=\{\mathbf{x}\in W(\lambda):\mathbf{x}(u)=0\}$ and $W_v(\lambda)=\{\mathbf{x}\in W(\lambda):\mathbf{x}(v)=0\}$.
\item Assume $P_X$ has pairwise orthogonal columns. Vertices $u$ and $v$ are strongly cospectral if and only if for each $\lambda\in\sigma_u$, we have
(i) $W_u(\lambda)=W_v(\lambda)$ and (ii) only one of $\{\mathbf{x}\in W(\lambda):\mathbf{x}(u)\mathbf{x}(v)=-1\}$ and $\{\mathbf{x}\in W(\lambda):\mathbf{x}(u)\mathbf{x}(v)=1\}$ is empty. Moreover, if $u$ and $v$ are strongly cospectral, then 
\begin{equation*}
\sigma_{uv}^+=\{\lambda:\mathbf{x}(u)\mathbf{x}(v)=1\ \text{for all}\ \mathbf{x}\in W(\lambda)\}
\end{equation*}
and
\begin{equation*}
\sigma_{uv}^-=\{\lambda:\mathbf{x}(u)\mathbf{x}(v)=-1\ \text{for all}\ \mathbf{x}\in W(\lambda)\}.
\end{equation*}
\end{enumerate}

\end{lemma}  

\begin{proof}
Let $\lambda$ be an eigenvalue of $L(X)$ with $W(\lambda)=\{\mathbf{x}_1,\ldots,\mathbf{x}_m\}$. Orthogonalizing $W(\lambda)$ yields an orthogonal basis of eigenvectors $\widetilde{W}(\lambda)=\{\widetilde{\mathbf{x}}_1,\ldots,\widetilde{\mathbf{x}}_m\}$ for the eigenspace associated with $\lambda$, and so $E_{\lambda}=\sum_{j=1}^m\frac{1}{\|\widetilde{\mathbf{x}}_{j}\|^2}\widetilde{\mathbf{x}}_{j}\widetilde{\mathbf{x}}_{j}^T$. This gives us
\begin{equation}
\label{cosp}
(E_{\lambda})_{u,u}=\sum_{j=1}^m\frac{\widetilde{\mathbf{x}}_{j}(u)^2}{\|\widetilde{\mathbf{x}}_{j}\|^2}.
\end{equation}
As $E_{\lambda}\mathbf{e}_u\neq 0$ if and only if $(E_{\lambda})_{u,u}>0$, (\ref{cosp}) implies that $\lambda\in\sigma_u$ if and only if $\widetilde{\mathbf{x}}_{j}(u)\neq 0$ for some $j$. Now, if $\lambda\in\sigma_u$ and $\ell$ is the smallest index such that $\widetilde{\mathbf{x}}_{\ell}(u)\neq 0$, then $\widetilde{\mathbf{x}}_{j}(u)=0$ for each $j\in\{1,\ldots,\ell\}$, and so $\mathbf{x}_{j}(u)=0$ for each $j\in\{1,\ldots,\ell\}$. Consequently, $\mathbf{x}_{j+1}(u)=\widetilde{\mathbf{x}}_{j+1}(u)\neq 0$, and so (1) holds.

The assumption in (2) implies that $W(\lambda)$ is now an orthogonal basis of eigenvectors for the eigenspace associated with $\lambda$. Since $\mathbf{x}(u)^2=1$ whenever $\mathbf{x}\not\in W_u(\lambda)$ and $\mathbf{x}(v)^2=1$ whenever $\mathbf{x}\not\in W_v(\lambda)$, we get
\begin{equation*}
\begin{split}
(E_{\lambda})_{u,u}-(E_{\lambda})_{v,v}&\stackrel{(*)}{=}\sum_{\mathbf{x}\in W_v(\lambda)\backslash W_u(\lambda)}\frac{\mathbf{x}(u)^2
}{\|\mathbf{x}\|^2}-\sum_{\mathbf{x}\in W_u(\lambda)\backslash W_v(\lambda)}\frac{\mathbf{x}(v)^2}{\|\mathbf{x}\|^2}+\sum_{\mathbf{x}\notin W_u(\lambda)\cup W_v(\lambda)}\frac{\mathbf{x}(u)^2-\mathbf{x}(v)^2}{\|\mathbf{x}\|^2}\\
&\stackrel{(**)}{=}\sum_{\mathbf{x}\in W_v(\lambda)\backslash W_u(\lambda)}\frac{1}{\|\mathbf{x}\|^2} -\sum_{\mathbf{x}\in W_u(\lambda)\backslash W_v(\lambda)}\frac{1}{\|\mathbf{x}\|^2}.
\end{split}
\end{equation*}
If $W_u(\lambda)=W_v(\lambda)$, then $(*)$ above yields $(E_{\lambda})_{u,u}-(E_{\lambda})_{v,v}=\sum_{\mathbf{x}\notin W_u(\lambda)\cup W_v(\lambda)}\frac{\mathbf{x}(u)^2-\mathbf{x}(v)^2}{\|\mathbf{x}\|^2}=0$, which results in cospectrality between $u$ and $v$. Otherwise, $u$ and $v$ are cospectral if and only if the right hand side of $(**)$ is zero. This proves (2). Finally, one can check that the conditions in (3) are equivalent to those in Lemma \ref{SC}, which guarantees strong cospectrality between $u$ and $v$.
\end{proof}

\begin{figure}[h]
\centering
\begin{subfigure}{.32\textwidth}
  \centering
  \begin{tikzpicture}
  [scale=0.95,auto=left,
  ns/.style={circle,fill=white!20,draw=black,minimum size=0.9cm,line width = 1pt,scale=0.65},
  rec/.style={rectangle,fill=gray!25,draw=black,minimum size=0.9cm,line width = 1pt,scale=0.65},
  st/.style={star, star points = 6,fill=gray!25,draw=black,minimum size=1cm,line width = 1pt,scale=0.65},
  blank/.style={circle,fill=white!20,minimum size=2mm},
  es/.style={draw=black, line width = 1pt},
  dash/.style={draw=black},
  every loop/.style={}
   ]
  \node[rec] (n3) at (0.6,3.4) {3};
  \node[ns] (n1) at (2,4) {1};
  \node[ns] (n2) at (3.4,3.4) {2};
  \node[rec] (n4) at (4,2) {4};
  \node[st] (n6) at (3.4,0.6) {6};
  \node[ns] (n8) at (2,0) {8};
  \node[ns] (n7) at (0.6,0.6) {7};
  \node[st] (n5) at (0,2) {5};
  
  \draw[es]  (n3) edge node{} (n4);
  \draw[es]  (n3) edge node{} (n6);
  \draw[es]  (n3) edge node{} (n8);
  \draw[es]  (n3) edge node{} (n7);
  \draw[es]  (n3) edge node{} (n5);
  
  \draw[es]  (n1) edge node{} (n6);
  \draw[es]  (n1) edge node{} (n8);
  \draw[es]  (n1) edge node{} (n7);
  \draw[es]  (n1) edge node{} (n5);

  \draw[es]  (n2) edge node{} (n6);
  \draw[es]  (n2) edge node{} (n8);
  \draw[es]  (n2) edge node{} (n7);
  \draw[es]  (n2) edge node{} (n5);

  \draw[es]  (n4) edge node{} (n6);
  \draw[es]  (n4) edge node{} (n8);
  \draw[es]  (n4) edge node{} (n7);
  \draw[es]  (n4) edge node{} (n5);

  \draw[es]  (n5) edge node{} (n6);
  \end{tikzpicture}
  \caption{}
  \label{fig:sub2}
\end{subfigure}
\begin{subfigure}{.32\textwidth}
  \centering
\begin{tikzpicture}
  [scale=0.95,auto=left,
  ns/.style={circle,fill=white!20,draw=black,minimum size=0.9cm,line width = 1pt,scale=0.65},
  rec/.style={rectangle,fill=gray!25,draw=black,minimum size=0.9cm,line width = 1pt,scale=0.65},
  st/.style={star, star points = 6,fill=gray!25,draw=black,minimum size=1cm,line width = 1pt,scale=0.65},
  cl/.style={cloud,fill=gray!25,draw=black,minimum size=1cm,line width = 1pt,scale=0.65},
  blank/.style={circle,fill=white!20,minimum size=2mm},
  es/.style={draw=black, line width = 1pt},
  dash/.style={draw=black},
  every loop/.style={}
   ]
  \node[rec] (n3) at (0.6,3.4) {3};
  \node[ns] (n1) at (2,4) {1};
  \node[ns] (n2) at (3.4,3.4) {2};
  \node[rec] (n4) at (4,2) {4};
  \node[cl] (n8) at (3.4,0.6) {8};
  \node[cl] (n7) at (2,0) {7};
  \node[st] (n6) at (0.6,0.6) {6};
  \node[st] (n5) at (0,2) {5};
  
  \draw[es]  (n3) edge node{} (n4);
  \draw[es]  (n3) edge node{} (n8);
  \draw[es]  (n3) edge node{} (n7);
  \draw[es]  (n3) edge node{} (n6);
  \draw[es]  (n3) edge node{} (n5);
  
  \draw[es]  (n1) edge node{} (n8);
  \draw[es]  (n1) edge node{} (n7);
  \draw[es]  (n1) edge node{} (n6);
  \draw[es]  (n1) edge node{} (n5);

  \draw[es]  (n2) edge node{} (n8);
  \draw[es]  (n2) edge node{} (n7);
  \draw[es]  (n2) edge node{} (n6);
  \draw[es]  (n2) edge node{} (n5);

  \draw[es]  (n4) edge node{} (n8);
  \draw[es]  (n4) edge node{} (n7);
  \draw[es]  (n4) edge node{} (n6);
  \draw[es]  (n4) edge node{} (n5);

  \draw[es]  (n8) edge node{} (n6);
  \draw[es]  (n8) edge node{} (n5);

  \draw[es]  (n7) edge node{} (n6);
  \draw[es]  (n7) edge node{} (n5);

\end{tikzpicture}
  \caption{}
  \label{fig:sub3}
\end{subfigure}
\begin{subfigure}{.32\textwidth}
  \centering
\begin{tikzpicture}
  [scale=0.95,auto=left,
  ns/.style={circle,fill=gray!25,draw=black,minimum size=0.9cm,line width = 1pt,scale=0.65},
  rec/.style={rectangle,fill=gray!25,draw=black,minimum size=0.9cm,line width = 1pt,scale=0.65},
  st/.style={star, star points = 6,fill=gray!25,draw=black,minimum size=1cm,line width = 1pt,scale=0.65},
  cl/.style={cloud,fill=gray!25,draw=black,minimum size=1cm,line width = 1pt,scale=0.65},
  blank/.style={circle,fill=white!20,minimum size=2mm},
  es/.style={draw=black, line width = 1pt},
  dash/.style={draw=black},
  every loop/.style={}
   ]
  \node[rec] (n3) at (0.6,3.4) {3};
  \node[ns] (n1) at (2,4) {1};
  \node[ns] (n2) at (3.4,3.4) {2};
  \node[rec] (n4) at (4,2) {4};
  \node[st] (n8) at (3.4,0.6) {8};
  \node[cl] (n7) at (2,0) {7};
  \node[cl] (n6) at (0.6,0.6) {6};
  \node[st] (n5) at (0,2) {5};
  
  \draw[es]  (n3) edge node{} (n4);
  \draw[es]  (n3) edge node{} (n7);
  \draw[es]  (n3) edge node{} (n5);
  \draw[es]  (n3) edge node{} (n6);
  \draw[es]  (n3) edge node{} (n8);
  
  \draw[es]  (n1) edge node{} (n2);
   \draw[es]  (n1) edge node{} (n8);
  \draw[es]  (n1) edge node{} (n5);
  \draw[es]  (n1) edge node{} (n7);
  \draw[es]  (n1) edge node{} (n6);

  \draw[es]  (n2) edge node{} (n8);
  \draw[es]  (n2) edge node{} (n5);
  \draw[es]  (n2) edge node{} (n7);
  \draw[es]  (n2) edge node{} (n6);

  \draw[es]  (n4) edge node{} (n8);
  \draw[es]  (n4) edge node{} (n5);
  \draw[es]  (n4) edge node{} (n7);
  \draw[es]  (n4) edge node{} (n6);

  \draw[es]  (n7) edge node{} (n6);
  \draw[es]  (n8) edge node{} (n5);
  \draw[es]  (n8) edge node{} (n6);
  \draw[es]  (n7) edge node{} (n5);

\end{tikzpicture}
  \caption{}
  \label{fig:sub4}
\end{subfigure}
\caption{The graphs (a) $(O_2\sqcup K_2)\vee (O_2\sqcup K_2)$, (b) $(O_2\sqcup K_2)\vee C_4$ and (c) $(K_2\sqcup K_2)\vee C_4$.}
\label{fig:test1}
\end{figure}
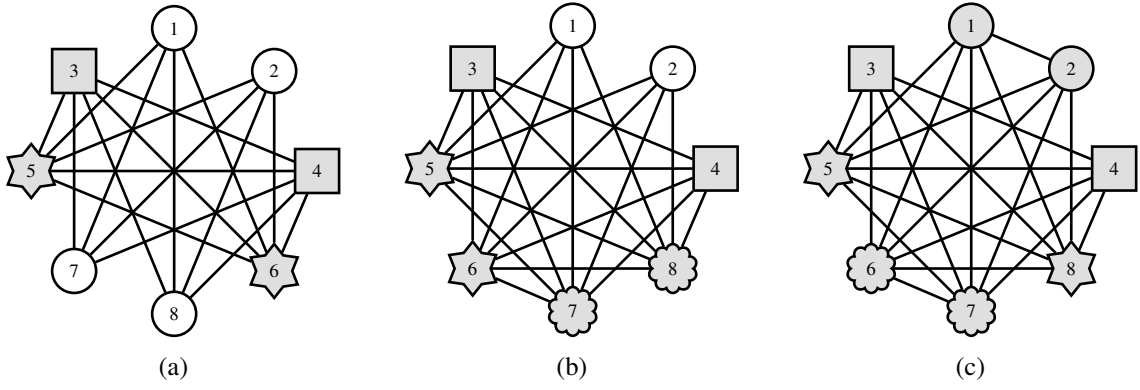

\begin{remark}
If $X$ is a Hadamard diagonalizable graph, i.e., $P_X$ is a Hadamard matrix, then, for each eigenvalue $\lambda$ of $L(X)$, we have $W_1(\lambda)=W_2(\lambda)=\varnothing$ because Hadamard matrices do not have zero entries. Invoking Lemma \ref{WHDSC}, we get that any two vertices in a Hadamard diagonalizable graph are cospectral.
\end{remark}

Lemma \ref{WHDSC}(2) can be used to determine cospectral vertices that are not strongly cospectral. Indeed, any pair of cospectral vertices satisfying (\ref{cosp1}) for some eigenvalue $\lambda$ of $L(X)$ are not strongly cospectral. Moreover, as cospectral vertices have the same eigenvalue supports, if $X$ has $\lambda$ as a simple eigenvalue with associated eigenvector $\mathbf{x}$, then $u$ and $v$ are not cospectral whenever $\mathbf{x}(u)\neq 0$ and $\mathbf{x}(v)=0$.

The following example concretely illustrates Lemma \ref{WHDSC}(2-3).

\begin{example}
\label{cosp}
Consider the graphs (a) $(O_2\sqcup K_2)\vee (O_2\sqcup K_2)$, (b) $(O_2\sqcup K_2)\vee C_4$, (c) $(K_2\sqcup K_2)\vee C_4$ in Figure \ref{fig:test1} and (d) $K_8\backslash e$ with $e=[1,2]$. These are not Hadamard diagonalizable, but their Laplacians are diagonalized by the weak Hadamard matrix $R=[\mathbf{x}_1\ \mathbf{x}_2\ \mathbf{x}_3\ \mathbf{x}_4\ \mathbf{x}_5\ \mathbf{x}_6\ \mathbf{x}_7\ \mathbf{x}_8]$ in Example \ref{eight}. From rows 19, 11, 12, and 5 of Table \ref{tab}, the associated eigenvalues for columns $\mathbf{x}_1,\ldots,\mathbf{x}_8$ of $R$ are $0,8,4,4,6,4,4,6$ for (a), $0,8,4,4,6,6,8,6$ for (b), $0,8,6,4,6,6,8,6$ for (c) and $0,8,6,8,8,8,8,8$ for (d), respectively. In (a), $\mathbf{x}_3$ and $\mathbf{x}_4$ are eigenvectors associated with the eigenvalue 4 satisfying $\mathbf{x}_3(1)=-\mathbf{x}_3(2)$ and $\mathbf{x}_4(1)=\mathbf{x}_4(2)$. Since this violates Lemma \ref{WHDSC}(3ii), vertices 1 and 2 in (a) are not strongly cospectral. The same argument can be used to check that vertices 7 and 8 in (a), and vertices 1 and 2 in (b), and any pair of vertices in $\{3,\ldots,8\}$ in (d) are not strongly cospectral. One can then check by inspection that for (a), vertices 3 and 4 (shaded squares), and 5 and 6 (shaded stars) are strongly cospectral with $\sigma_{j,j+1}^+(L)=\{0,4,8\}$ and $\sigma_{j,j+1}^-(L)=\{6\}$ for each $j\in\{3,5\}$. For (b), vertices 3 and 4 (shaded squares), 5 and 6 (shaded stars) and 7 and 8 (shaded clouds) are strongly cospectral with $\sigma_{3,4}^+(L)=\{0,4,8\}$ and $\sigma_{3,4}^-(L)=\{6\}$, and $\sigma_{j,j+1}^+(L)=\{0,8\}$ and $\sigma_{j,j+1}^-(L)=\{6\}$ for $j\in\{5,7\}$. For (c), vertices 1 and 2 (shaded circles), 3 and 4 (shaded squares), 5 and 6 (shaded stars) and 7 and 8 (shaded clouds) are strongly cospectral $\sigma_{j,j+1}^+(L)=\{0,4,8\}$ and $\sigma_{j,j+1}^-(L)=\{6\}$ for $j\in\{1,3\}$, and $\sigma_{j,j+1}^+(L)=\{0,8\}$ and $\sigma_{j,j+1}^-(L)=\{6\}$ for $j\in\{5,7\}$. Lastly, for (d), vertices 1 and 2 are strongly cospectral with $\sigma_{1,2}^+(L)=\{0,8\}$ and $\sigma_{1,2}^-(L)=\{6\}$.

In (d), if $u\neq 1,2$, then $\sigma_u=\{0,8\}$. Since $W_3(8)=W_4(8)=\{\mathbf{x}_6,\mathbf{x}_7,\mathbf{x}_8\}$, Lemma \ref{WHDSC}(2) implies that vertices 3 and 4 are cospectral. Furthermore, $W_3(8)\backslash W_5(8)=\{\mathbf{x}_4,\mathbf{x}_5\}$ and $W_5(8)\backslash W_3(8)=\{\mathbf{x}_6,\mathbf{x}_7\}$, and since $\|x_4\|=\|x_7\|$ and $\|x_5\|=\|x_6\|$, Equation~(\ref{cosp1}) holds, and so vertices 3 and 5 are cospectral by Lemma \ref{WHDSC}(2). The same argument shows that 3 is cospectral with 6, 7 and 8. Thus, while no pair of vertices in (d) are strongly cospectral except for 1 and 2, any two vertices in $\{3,\ldots,8\}$ in (d) 
are cospectral. 
\end{example}

The next example yields an infinite family of WHD graphs that are not Hadamard diagonalizable containing strongly cospectral vertices.

\begin{example}
\label{sckne}
While $K_n\backslash e$ is WHD for all $n\geq 4$ \cite[Lemma 4.8]{adm2021}, any weak Hadamard matrix diagonalizing $L(K_n\backslash e)$ need not have pairwise orthogonal columns whenever $n\not\equiv 0$ (mod 4). Nonetheless, the non-adjacent vertices of $K_n\backslash e$ are strongly cospectral \cite[Corollary 6.9(2)]{MonterdeELA}, and these two vertices are the only strongly cospectral pair in $K_n\backslash e$ for all $n\geq 3$.
\end{example}

\subsection{Perfect state transfer}

We now characterize PST in WHD graphs with $P_X$ having pairwise orthogonal columns. Some simple Python code implementing Theorem~\ref{PST1}  is available for download at \cite{sup}. 


\begin{theorem}
\label{PST1}
Let $X$ be a weighted graph. Suppose $S=[\mathbf{x}_1,\ldots,\mathbf{x}_n]$ has pairwise orthogonal columns and $L(X)=SDS^{-1}$, where $D=\text{diag}(\lambda_1,\ldots,\lambda_n)$. Then perfect state transfer occurs between two vertices $u$ and $v$ in $X$ at time $\tau$ if and only if for each eigenvalue $\lambda_j$ of $L(X)$,
\begin{equation}
\label{charPST}
e^{i\tau\lambda_j}\mathbf{x}_j(u)=\mathbf{x}_j(v).
\end{equation}
If we add that $X$ is WHD, then we may write (\ref{charPST}) as
\begin{equation}
\label{charPST1}
e^{i\tau\lambda_j}=\mathbf{x}_j(u)\mathbf{x}_j(v).
\end{equation}
\end{theorem}

\begin{proof}
Since $L=SDS^{-1}$, we have $U_L(t)=Se^{it D}S^{-1}$. Recall that PST occurs between $u$ and $v$ at time $\tau$ if and only if $U_L(t)\mathbf{e}_u=\gamma \mathbf{e}_v$. Equivalently, $e^{i\tau D}S^{-1}\mathbf{e}_u=\gamma S^{-1}\mathbf{e}_v$ for some unit $\gamma\in\mathbb{C}$. Since $0$ is an eigenvalue of $L(X)$ with eigenvector $\mathbf{1}$, we have $0\in\sigma_{uv}^+$. Now, since $e^{i\tau \lambda}$ is a phase factor for PST for all $\lambda\in\sigma_{uv}^+$, we conclude that $\gamma=1$, and so 
\begin{equation}
\label{PST}
e^{i\tau D}S^{-1}\mathbf{e}_u=S^{-1}\mathbf{e}_v.
\end{equation}
Now, let $Q=\text{diag}(\|\mathbf{x}_1\|^2,\ldots,\|\mathbf{x}_n\|^2)$. Since $S$ has pairwise orthogonal columns, some algebra  yields $S^{-1}=Q^{-1}S^T$, and so $S^{-1}\mathbf{e}_u=Q^{-1}S^T\mathbf{e}_u$. Pre-multiplying (\ref{PST}) to the right by $\mathbf{e}_j^T$ then yields $\frac{e^{it\lambda_j}}{\|x_j\|^2}\mathbf{x}_j(u)=\frac{1}{\|x_j\|^2}\mathbf{x}_j(v)$, which completes the proof.
\end{proof}

One important consequence of Theorem \ref{PST1} is that it can be used to construct weighted WHD graphs having PST at some specified time $\tau>0$. Indeed, if we fix $\tau>0$ and a normalized weak Hadamard $S$ with pairwise orthogonal columns, then by choosing the eigenvalues in $D$ to satisfy (\ref{charPST1}), $L = SDS^{-1}$ will be the Laplacian matrix of some rational-weighted graph with PST at time $\tau$. One can then scale $L$ if one wishes to obtain an integer-weighted graph.

We denote the largest power of two that divides an integer $z$ by $\nu_2(z)$. It is known that every integer $z$ can be written uniquely as $z=2^{\nu_2(z)}a$, where $a$ is odd. Using Coutinho's characterization of PST \cite[Theorem 2.4.4]{Coutinho2014}, one can show that Theorem \ref{PST1} can be restated in the context of WHD graphs as follows.

\begin{corollary}
\label{PST2}
Let $X$ be an integer-weighted WHD graph, where $P_X$ has pairwise orthogonal columns. Then perfect state transfer occurs between vertices $u$ and $v$ if and only if these two conditions hold. 
\begin{enumerate}
\item Vertices $u$ and $v$ are strongly cospectral with
\begin{equation*}
\sigma_{uv}^+=\{\lambda:\mathbf{x}(u)\mathbf{x}(v)=1\ \text{for all}\ \mathbf{x}\in W(\lambda)\}\ \text{and}\ \sigma_{uv}^-=\{\lambda:\mathbf{x}(u)\mathbf{x}(v)=-1\ \text{for all}\ \mathbf{x}\in W(\lambda)\}.
\end{equation*}
\item $\nu_2(\lambda_{j})>\nu_2(\lambda_{k})=\nu_2(\lambda_{\ell})$ for all $\lambda_{k},\lambda_{\ell}\in \sigma_{uv}^-$ and for all $\lambda_j\in \sigma_{uv}^+$ with $\lambda_j>0$.
\end{enumerate}
Moreover, if PST occurs between $u$ and $v$, then the minimum time it occurs is $\frac{\pi}{g}$, where $g=\operatorname{gcd}(\lambda_1,\ldots,\lambda_n)$.
\end{corollary} 

\begin{remark}
\label{remtime}
Whether $X$ is WHD or not, condition (2) and the statement about the minimum PST time in Corollary \ref{PST2} holds as long as $X$ is Laplacian integral.
\end{remark}

Our calculations indicate that condition 1 of Corollary~\ref{PST2} is generally the condition that fails when there is no PST in a WHD graph. We also remark that every PST time is an odd multiple of $\frac{\pi}{g}$.

For Laplacian integral graphs, it is clear from Corollary \ref{PST2}(2) and Remark \ref{remtime} that PST does not occur whenever all non-zero Laplacian eigenvalues are odd. In particular, vertex $u$ is not involved in PST if all eigenvalues in $\sigma_u$ are odd. Furthermore, PST does not occur between strongly cospectral vertices $u$ and $v$ whenever some $\lambda_j\in\sigma_{uv}^+$ is odd or $\nu_2(\lambda_j)\leq \nu_2(\lambda_k)$ for some $\lambda_j\in\sigma_{uv}^+$ and $\lambda_k\in\sigma_{uv}^-$.

Now, suppose $X$ is a Hadamard diagonalizable integer-weighted graph and $u$ is a vertex of $X$. Then all eigenvalues of $L$ are even integers \cite[Theorem 2]{johnston2017}. Thus, if $u$ exhibits PST in $X$ with minimum PST time $\tau=\frac{\pi}{2}$, then Corollary \ref{PST2} implies that all elements in $\sigma_{uv}^+$ are integers $\lambda\equiv 0$ (mod 4), while all elements in $\sigma_{uv}^-$ are integers $\lambda\equiv 2$ (mod 4). This yields the number-theoretic condition required of the eigenvalues in the support of a vertex in  graph to exhibit PST, which was first established by Johnston et.\ al \cite[Theorem 3]{johnston2017}. In fact, Theorem \ref{PST1} generalizes \cite[Theorem 3]{johnston2017} in two ways. First, it allows for arbitrary PST time instead of just $\frac{\pi}{2}$, and second, it extends the result to WHD graphs, instead of just Hadamard diagonalizable graphs.

\begin{example}
\label{pstex}
Consider the graphs $(O_2\sqcup K_2)\vee (O_2\sqcup K_2)$, $(O_2\sqcup K_2)\vee C_4$, $(K_2\sqcup K_2)\vee C_4$ and $K_8\backslash e$ with $e=[1,2]$ in Example \ref{cosp} whose Laplacians are diagonalized by the matrix $R$. Since $R$ has pairwise orthogonal columns, a direct application of Corollary~\ref{PST2} to these graphs shows that they all exhibit PST between strongly cospectral pairs of vertices with minimum time $\tau=\pi/2$. In $(O_2\sqcup K_2)\vee (O_2\sqcup K_2)$, these are vertices $j$ and $j+1$ for each $j\in\{3,5\}$; in $(O_2\sqcup K_2)\vee C_4$, these are vertices $j$ and $j+1$ for each $j\in\{3,5,7\}$; and in $(K_2\sqcup K_2)\vee C_4$, these are vertices $j$ and $j+1$ for each $j\in\{1,3,5,7\}$ (see graphs (a), (b) and (c) in Figure \ref{fig:test1}). Moreover, in $K_8\backslash e$, these are vertices $1$ and $2$. It is important to note that  $(O_2\sqcup K_2)\vee (O_2\sqcup K_2)$, $(O_2\sqcup K_2)\vee C_4$, and $K_8\backslash e$ are not regular and therefore not Hadamard diagonalizable, thus providing examples of bonafide weak Hadamard diagonalizable graphs having PST.
\end{example}

From Example \ref{pstex}, we observe that $O_4\vee K_4$, $K_8\backslash e$, $(O_2\sqcup K_2)\vee (O_2\sqcup K_2)$, $(O_2\sqcup K_2)\vee C_4$, and $(K_2\sqcup K_2)\vee C_4$ are WHD graphs that are not Hadamard diagonalizable having exactly zero, one, two, three, and four pairs of vertices that exhibit PST, respectively. This intriguing observation leads us to suspect that some WHD graphs are excellent sources of PST. As an initial investigation, we examined PST in WHD graphs on eight vertices whose diagonalizing weak Hadamard matrix has pairwise orthogonal columns (see the last column of Table \ref{tab} for a summary). As it turns out, there are only 3 amongst 23 such graphs that do not exhibit PST, namely $K_8$, $K_{4,4}$ and $O_4\vee K_4$. This suggests that in general, the subclass of WHD graphs with the property that the diagonalizing weak Hadamard matrix has pairwise orthogonal columns provides excellent sources of PST. Moreover, it is interesting to note that amongst the 20 graphs that exhibit PST, the vertices that pair up to exhibit PST are the same pairs that exhibit strong cospectrality.

\begin{table}
\begin{center}
\begin{tabularx}{1\textwidth} { 
  | >{\centering\arraybackslash\hsize=.1\hsize\linewidth=\hsize}X 
  | >{\raggedright\arraybackslash\hsize=1.2\hsize\linewidth=\hsize}X 
  | >{\centering\arraybackslash\hsize=.8\hsize\linewidth=\hsize}X 
  | >{\centering\arraybackslash\hsize=.3\hsize\linewidth=\hsize}X 
  | >{\centering\arraybackslash\hsize=.9\hsize\linewidth=\hsize}X
  | >{\centering\arraybackslash\hsize=.7\hsize\linewidth=\hsize}X | }
 \hline
	 & \ \ \ \ \ \ \ \ \ \ Graph & Diagonalized by & HD & Spectrum &  \# of PST pairs \\ \hline
  
     \text{1}& $O_4\vee K_4$ & R & No & $0,8,4,4,4,8,8,8$ & 0 \\ \hline
     \text{2}& $(O_2\sqcup K_2)\vee K_4$ & R & No & $0,8,4,4,6,8,8,8$ & 1 \\ \hline
     \text{3}& $(K_2\sqcup K_2)\vee K_4$ & R & No & $0,8,6,4,6,8,8,8$ & 2 \\ \hline
     \text{4}& $C_4\vee K_4$ & R & No & $0,8,6,8,6,8,8,8$ & 2 \\ \hline
     \text{5}& $K_4\backslash e\vee K_4\cong K_8\backslash e$ & R & No & $0,8,6,8,8,8,8,8$ & 1 \\ \hline
     \text{6}& $K_4\vee K_4\cong K_8 $ & R & Yes & $0,8,8,8,8,8,8,8$ & 0 \\ \hline
     \text{7}& $O_4\vee K_4\backslash e$ & R & No & $0,8,4,4,4,6,8,8$ & 1 \\ \hline
     \text{8}& $(O_2\sqcup K_2)\vee K_4\backslash e$ & R & No & $0,8,4,4,6,6,8,8$ & 2 \\ \hline
     \text{9}&  $(K_2\sqcup K_2)\vee K_4\backslash e$ & R & No & $0,8,6,4,6,6,8,8$ & 3 \\ \hline
     \text{10}& $O_4\vee C_4$ & R & No & $0,8,4,4,4,6,8,6$ & 2 \\ \hline
     \text{11}& $(O_2\sqcup K_2)\vee C_4$ & R & No & $0,8,4,4,6,6,8,6$ & 3 \\ \hline
     \text{12}& $(K_2\sqcup K_2)\vee C_4$ & R & No & $0,8,6,4,6,6,8,6$ & 4 \\ \hline
     \text{13}& $ C_4\vee C_4\cong K_{2,2,2,2}$ & R & Yes & $0,8,6,8,6,6,8,6$ & 4 \\ \hline
     \text{14}& $K_4\backslash e\vee C_4$ & R & No & $0,8,6,8,8,6,8,6$ & 3 \\ \hline
     \text{15}& $O_4\vee (K_2\sqcup K_2)$ & R & No & $0,8,4,4,4,6,4,6$ & 2 \\ \hline
     \text{16}& $(O_2\sqcup K_2)\vee (K_2\sqcup K_2)$ & R & No & $0,8,4,4,6,6,4,6$ & 3 \\ \hline
     \text{17}& $(K_2\sqcup K_2)\vee (K_2\sqcup K_2)$ & R & Yes & $0,8,6,4,6,6,4,6$ & 4 \\ \hline
     \text{18}& $O_4\vee (O_2\sqcup K_2)$ & R & No & $0,8,4,4,4,4,4,6$ & 1 \\ \hline
     \text{19}& $(O_2\sqcup K_2)\vee (O_2\sqcup K_2)$ & R & No & $0,8,4,4,6,4,4,6$ & 2 \\ \hline
     \text{20}& $O_4\vee O_4\cong K_{4,4}$ & R & Yes & $0,8,4,4,4,4,4,4$ & 0 \\ \hline
     \text{21}& $O_2\odot K_4$ (hypercube) & T & Yes & $0,4,4,4,6,2,2,2$ & 4 \\ \hline
     \text{22}& $K_2\square K_4\cong (O_2\odot K_4)^c$  & T & Yes & $0,4,4,4,2,6,6,6$  & 4 \\ \hline
     \text{23}& $(O_2\sqcup K_2) \odot K_4$ & T & No & $0,4,4,6,6,2,2,4$ & 2 \\ \hline
\end{tabularx}
\caption{\label{tab} WHD graphs on 8 vertices whose Laplacian matrices are diagonalized by a weak Hadamard with pairwise orthogonal columns, the corresponding weak Hadamard matrix that diagonalizes them ($R$ and $T$ are matrices given in Example \ref{eight}), whether they are Hadamard diagonalizable (HD), their corresponding spectra $\lambda_1,\ldots,\lambda_8$ with the $j$th column of either $R$ or $T$ as the associated eigenvector of $\lambda_j$ for each $j$, and their corresponding number of pairs of vertices that exhibit PST}
\end{center}
\end{table}

\subsection{Complements and joins}

The following result provides a sufficient condition for PST to occur in the complement of a Laplacian integral connected graph.

\begin{proposition}
\label{comp1}
Let $X$ be a connected unweighted graph such that $L(X)$ is diagonalized by a matrix $S$.
\begin{enumerate}
\item Vertices $u$ and $v$ in $X$ are strongly cospectral in $X$ if and only if they are strongly cospectral in $X^c$.
\item Suppose $X$ is a Laplacian integral graph and $S$ has pairwise orthogonal columns. If perfect state transfer occurs between $u$ and $v$ in $X$ at $\tau=\frac{\pi}{g}$, where $g$ is given in Corollary \ref{PST2}, then perfect state transfer occurs between $u$ and $v$ in $X^c$ at time $\tau$ if and only if $\frac{n}{g}$ is even.
\end{enumerate}
\end{proposition}

\begin{proof}
Invoking Proposition \ref{comp} yields (1). To show (2), let $S=[\mathbf{1},\mathbf{x}_2,\ldots,\mathbf{x}_n]$, and $\lambda_1=0,\lambda_2,\ldots,\lambda_n$ be the eigenvalues of $L(X)$ with corresponding eigenvectors $\mathbf{1},\mathbf{x}_2,\ldots,\mathbf{x}_n$. The eigenvalues of $L(X^c)$ are $0,n-\lambda_2,\ldots,n-\lambda_n$, where $n-\lambda_j$ and $\lambda_j$ have the same eigenvectors for each $j\geq 2$. Since PST occurs between $u$ and $v$ at time $\tau$, Theorem \ref{PST1} guarantees that $e^{i\tau\lambda_j}\mathbf{x}_j(u)=\mathbf{x}_j(v)$ for all $j\geq 2$. Now, observe that $e^{i\tau(n-\lambda_j)}\mathbf{x}(u)=\mathbf{x}(j)$ holds if and only if $e^{i\tau n}=1$. Since $X$ is Laplacian integral, $\tau=\frac{\pi}{g}$ from Remark \ref{remtime}. Thus, $e^{i\tau n}=1$ if and only if $\frac{n}{g}$ is even.
\end{proof}

To illustrate Proposition \ref{comp1}(2), consider $X=K_1\sqcup K_2$, which is a disconnected WHD graph with $P_X$ having pairwise orthogonal columns. Note that $K_2$ exhibits PST at $\frac{\pi}{2}$. Since the Laplacian eigenvalues of $X$ are $0$ and $2$, it follows that $g=2$, and so $\frac{n}{g}=\frac{3}{2}$ is not even. Consequently, $(K_1\sqcup K_2)^c=P_3$ does not exhibit PST between end vertices, which is a well-known result in quantum state transfer.

Consider again $Z_k=X\vee \ldots \vee X$, which is the $k$-fold join of $X$ with itself. Denote by $(u,j)$ the copy of $u\in V(X)$ in the $j$th copy of $X$ in $Z_k$. Our next result provides a sufficient condition for strong cospectrality and PST to occur in $Z_k$.

\begin{proposition}
\label{QSTjoin}
Let $X$ be a weighted graph.
\begin{enumerate}
    \item If $u$ and $v$ are strongly cospectral in $X$, then $(u,j)$ and $(v,j)$ are strongly cospectral in $Z_k$ with
\begin{equation*}
\sigma_{(u,j),(v,j)}^+=\{\lambda+(k-1)n:0<\lambda\in\sigma_{uv}^+(L(X))\}\cup\{0,kn\}
\end{equation*}  
and
\begin{equation*}
\sigma_{(u,j),(v,j)}^-=\{\lambda+(k-1)n:\lambda\in \sigma_{uv}^-(L(X))\}.
\end{equation*}   
\item Let $X$ be a Laplacian integral graph such that $L(X)$ is diagonalized by a matrix $S$ with pairwise orthogonal columns. If perfect state transfer occurs between vertices $u$ and $v$ in $X$, then it occurs between vertices $(u,j)$ and $(v,j)$ in $Z_k$ if and only if $\frac{n}{g}$ is even, where $g$ is given in Corollary \ref{PST2}.
\end{enumerate}
\end{proposition}

\begin{proof}

Invoking \cite[Lemma 4]{kirkland2023quantum} yields (1). To prove (2), let $0,\lambda_2,\ldots,\lambda_n$ be the eigenvalues of $L(X)$. Then by Theorem \ref{thm:join}, $0,\lambda_2+(k-1)n,\ldots,\lambda_n+(k-1)n,kn$ are the eigenvalues of $L(Z_k)$, Since $kn\in \sigma_{(u,j),(v,j)}^+$, applying the same argument in the proof of Proposition \ref{comp1} yields the desired result.
\end{proof}

\begin{remark}
Propositions \ref{comp1} and \ref{QSTjoin} both apply to WHD graphs with $P_X$ having orthogonal columns. In particular, if $X$ is unweighted Laplacian integral and either (i) $n$ and $g$ are powers of two or (ii) $n$ is doubly even and $g=2$, then we get that $\frac{n}{g}$ is even. In this case, PST in $X$ guarantees PST in $Z_k$ and $X^c$.
\end{remark}

We end this section by providing infinite families of WHD graphs that exhibit PST. Recall that both $C_4$ and $K_2\sqcup K_2$ exhibit PST with minimum time $\pi/2$ between two pairs of vertices, both $K_4\backslash e$ and $O_2\sqcup K_2$ exhibit PST with minimum time $\pi/2$ between a pair of vertices, and $K_4$ and $O_4$ do not exhibit PST.

\begin{example}
Let $k=2^{\ell}$ and $X_{(k)}=\bigsqcup_{j=1}^kX_j$, where $X_j\in\{K_4,C_4,K_2\sqcup K_2,O_4,K_4\backslash e,O_2\sqcup K_2\}$. By Example \ref{uncomp}, $\mathcal{F}=\{X_{(k)}:k=2^{\ell},\ell\geq 1\}$ is an infinite family of disconnected WHD graphs, where $X_{(k)}^c$ is also WHD with $P_{X_{(k)}}=P_{X_{(k)}^c}$ having pairwise orthogonal columns. The following hold.
\begin{enumerate}
\item If $X_j\in\{C_4,K_2\sqcup K_2,K_4\backslash e,O_2\sqcup K_2\}$ for at least one $j$, then $X_{(k)}$ exhibits PST with minimum time $\pi/2$. Since $g=2$ and $n$ is even, Proposition \ref{comp1} implies that $X_{(k)}^c$ exhibits PST with minimum time $\pi/2$ between the same pairs of vertices. In particular, the number of pairs of vertices in $X_{(k)}$ that exhibit PST is $2a+b$, where $a$ is the number of copies of $C_4$ and $K_2\sqcup K_2$ in $X_{(k)}$ and $b$ is the number of copies of $K_4\backslash e$ and $O_2\sqcup K_2$ in $X_{(k)}$. 
\item If $X_j\in\{K_4,O_4\}$ for each $j$, then $X_{(k)}$ does not exhibit PST.
\end{enumerate}
Moreover, neither $X_{(k)}$ nor $X_{(k)}^c$ are Hadamard diagonalizable whenever at least two $X_j$'s are distinct. Hence, if each $X_{(k)}\in \mathcal{F}$ has the property that $X_j\in\{C_4,K_2\sqcup K_2,K_4\backslash e,O_2\sqcup K_2\}$ for at least one $j$ and at least two $X_j$'s are distinct, then $\{X_{(k)}^c:k=2^{\ell},\ell\geq 1\}$ is an infinite family of WHD graphs that exhibit PST whereby each $X_{(k)}^c$ is not Hadamard diagonalizable but $P_{X_{(k)}^c}$ has pairwise orthogonal columns.

\end{example}

\begin{example}
Let $k=2^{\ell}$ and suppose $X\in\{C_4,K_2\sqcup K_2,K_4\backslash e,O_2\sqcup K_2\}$. From Example \ref{n=4}, each $Z_k$ is WHD and $P_{Z_k}$ has pairwise orthogonal columns. Since $n=4$ and $g=2$, a direct application of Proposition \ref{QSTjoin}(2) yields PST in each $Z_k$ at time $\pi/2$. Moreover, the number of pairs vertices that exhibits PST in $Z_k$ is $2k$ whenever $X\in \{C_4,K_2\sqcup K_2\}$, while it is $k$ whenever $X\in \{K_4\backslash e,O_2\sqcup K_2\}$. Moreover, if $X\in \{K_4\backslash e,O_2\sqcup K_2\}$, then $\{Z_k:k=2^{\ell},\ell\geq 1\}$ is an infinite family of WHD graphs that exhibit PST whereby each $Z_k$ is not Hadamard diagonalizable but $P_{Z_k}$ has pairwise orthogonal columns.
\end{example}

\begin{example}
Denote the transition matrix of a graph $Z$ by $U_{Z}(t)$. From \cite[Lemma 4.2]{Godsil2011StateTO}, it is known that $U_{X\square Y}(t)=U_X(t)\otimes U_Y(t)$. Consequently, PST occurs between $(u,w)$ and $(v,x)$ in $X\square Y$ at time $\tau$ whenever PST occurs between $u$ and $v$ in $X$ and $w$ and $x$ in $Y$ at time $\tau$.

Let $\square_{j=1}^kX_j:=X_1\square X_2\square\cdots\square X_k$. For each $k\geq 1$, let $Y_k=\square_{j=1}^kX_j$, where each $X_j\in\{C_4,K_2\sqcup K_2,K_4\backslash e,O_2\sqcup K_2\}$ and $X_j\in\{K_4\backslash e,O_2\sqcup K_2\}$ for at least one $j$. Since each $X_j$ exhibits PST, say between $u_j$ and $v_j$, we get that PST occurs in $Y_k$ between $(u_1,\ldots,u_k)$ and $(v_1,\ldots,v_k)$. As $P_{X_j}$ has pairwise orthogonal columns for each $j$, Theorem \ref{prod} implies that $\{Y_k:k\geq 1\}$ is an infinite family of WHD graphs that exhibit PST whereby each $Y_k$ is not Hadamard diagonalizable but $P_{Y_k}$ has pairwise orthogonal columns.
\end{example}

\section{Conclusion and Future Work}
This work uses the initial study of weak Hadamard matrices  in \cite{adm2021} as a stepping stone. First, we investigated algebraic and combinatorial properties of weak Hadamard matrices, providing numerous methods of constructing such matrices, and exploring  the idea of equivalent weak Hadamard matrices. We then turned our attention to weakly Hadamard diagonalizable graphs, providing numerous results for graph unions, complements, joins, and merges. Lastly, we explored quantum state transfer in WHD graphs in great detail, focusing on strong cospectrality and perfect state transfer, providing infinite families of WHD graph exhibiting PST, and providing numerous examples illustrating the power of our results through graph complements and joins. We provide Python code for some of the technical computations in \cite{sup}. 

The study of weak Hadamard matrices and WHD graphs is still very much in its infancy. We identify several open problems, the answer to which would help propel the study of these concepts forward. 

With respect to weak Hadamard matrices: Further to Theorem \ref{thm:EquivalentWeak}, it is an open problem  to characterize the number of equivalent weak Hadamard matrices attained through any operation (permutation of columns and/or rows, as well as negation of columns and/or rows). It would  be of interest to find the exact number of non-equivalent weak Hadamard matrices for reasonably small $n$ as well as the exact number of non-equivalent  weak Hadamard matrices with pairwise orthogonal columns, as well as the corresponding number of non-isomorphic graphs that these matrices diagonalize (a lower bound for such graphs for $n\leq 9$ can be inferred from \cite{adm2021}). 

An open problem pertinent to our work was recently posed \cite{over}: if $\textbf1$ together with $n-1$ non-zero vectors with entries  in $\{-1,0,1\}$ are mutually orthogonal, does it follow that $n\in \{1,2\}\cup \{4k\,:\, k\in \mathbb N\}$?   M.~Alekseyev has shown that the answer is to the affirmative for small $n$ (when $n\leq 12$), and I.~Bogdanov has shown the answer is to the affirmative when $n$ is an odd prime, but the general case remains elusive. If a graph can be diagonalized by a weak Hadamard matrix, then it can be diagonalized by a normalized weak Hadamard matrix. So, this open problem expands on Conjecture~\ref{conj:odd} to include any dimension $n\not\equiv 0$ mod 4, with $n>2$. 

For weakly Hadamard diagonalizable graphs, it would be of interest to compare, for reasonably small $n$,  the number of Hadamard diagonalizable graphs (this is known from \cite{breen} for $n\leq 36$), the number of weakly Hadamard diagonalizable graphs,  and, in particular, the number of weakly Hadamard diagonalizable graphs with $P_X$ having pairwise orthogonal columns. Although we have provided infinite families of WHD graphs, and methods of constructing WHD graphs, it would be useful to see just how prevalent these graphs are in comparison to Hadamard diagonalizable graphs.  

Finally, given the connection to quantum state transfer, it would be of interest to find the number of connected graphs that are Hadamard diagonalizable, weakly Hadamard diagonalizable, and weakly Hadamard diagonalizable  with $P_X$ having pairwise orthogonal columns,  that exhibit perfect state transfer, for relatively small $n$.

\section*{Acknowledgements}
 H.\ Monterde is supported by the University of Manitoba Faculty of Science and Faculty of Graduate Studies. S.\ Plosker is supported by NSERC Discovery grant number RGPIN-2019-05276, the Canada Research Chairs Program grant number 101062, and the  Canada Foundation for Innovation grant number 43948.

\bibliographystyle{plain}
\bibliography{references}

\end{document}